\documentclass[final,leqno,onefignum,onetabnum]{siamltexmm}

\usepackage{amsmath, amssymb}
\usepackage{enumerate}

\usepackage{soul}
\usepackage{color}
\definecolor{myorange}{rgb}{.8,.0,.0}
\newcommand{\myhl}[1]{{{{#1}}}}
\newcommand{\myhll}[1]{{{#1}}}

\newcommand{\mb}[1]{\mathbf{#1}}
\newcommand{\mbb}[1]{\mathbb{#1}}
\newcommand{\mcl}[1]{\mathcal{#1}}

\newcommand{\EE}{\mathbb{E\:}}
\newcommand{\VV}{\mathbf{Var}}
\newcommand{\reals}{\mathbb{R}}
\newcommand{\integers}{\mathbb{N}}

\newcommand{\dd}{\text{d}}

 \newtheorem{remark}{Remark}

\newtheorem{example}{Example}
 \newtheorem{assumption}{Assumption}
\title{Well-posed Bayesian Inverse Problems: Priors with Exponential Tails
  \thanks{This work was supported in part by the Natural Sciences and
    Engineering Council of Canada.}}

 \author{Bamdad Hosseini and Nilima Nigam
 \footnotemark[2]}

\begin{document}

    %\ead{bhossein@sfu.ca}
% \begin{indented}
% \item[]\today
% \end{indented}

\maketitle

 \newcommand{\slugmaster}{%
 \slugger{}{xxxx}{xx}{x}{x--x}}%slugger should be
%set to juq, siads, sifin, or siims
\renewcommand{\thefootnote}{\fnsymbol{footnote}}
\footnotetext[2]{Department of mathematics, Simon Fraser University,
  8888 University Drive, Burnaby, BC, V5A 1S6, Canada
  (\email{bhossein@sfu.ca},  \email{nigam@math.sfu.ca}).}
%\footnotetext[2]{Corresponding author}
\renewcommand{\thefootnote}{\arabic{footnote}}

\begin{abstract}
We consider the well-posedness of Bayesian inverse
problems when the prior measure has exponential tails. 
In particular, we consider the class of \myhl{\it convex (log-concave) probability measures} which
include the Gaussian and Besov measures as well as certain classes of
hierarchical priors.  
We identify appropriate conditions on the likelihood distribution and
the prior measure which guarantee
existence,
uniqueness and stability of the posterior measure with respect to
perturbations of the data. \myhl{We also consider consistent
approximations of the posterior such as discretization by projection.}
Finally, we present a general recipe for 
construction of convex priors on Banach spaces which will be of
interest in practical applications where one often works with spaces 
such as $L^2$ or the continuous functions.
\end{abstract}



\pagestyle{myheadings}
\thispagestyle{plain}
%\markboth{Convex priors for well-posed Bayesian inverse problems}

%\submitto{\IP}
%\maketitle

\section{Introduction}\label{sec:introduction}

Readers are likely familiar with the generic inverse problem: locate a
$u\in X$ from some data $y\in Y$ given \myhl{the model
 \begin{equation}
 \label{generci-inverse-problem}
   y = \tilde{\mcl{G}}(u),
 \end{equation}
where $\tilde{\mcl{G}}$ is a generic stochastic mapping referred to as the {\it parameter to observation map} that models the
relationship between the parameter and the observed data by taking the
measurement noise into account (be it additive, multiplicative
etc). Here $X$ and $Y$ are Banach spaces with norms $\|\cdot \|_X,\|\cdot \|_Y$ respectively.
 As an
example, if the measurement noise is additive then we can write
\begin{equation}
 \label{generci-inverse-problem-additive}
\tilde{\mcl{G}}(u) = \mcl{G}(u) + \eta,
\end{equation}}
\myhl{where $\mcl{G}$ is referred to as the {\it forward model} which is a deterministic mapping that associates $u$ to $y$.} Stated in this generality, of course, it is not at all obvious that we can find a solution to \eqref{generci-inverse-problem}, nor how we should locate it.

The Bayesian approach to solution of inverse problems of the form \eqref{generci-inverse-problem}
has attracted
much attention in the past decade \cite{somersalo,
  stuart-acta-numerica}. These  methods are well established
in the statistics literature \cite{calvetti, bernardo} where they are
often applied to problems where $u$ belongs to a finite-dimensional space.
 However, the Bayesian approach in the setting of infinite-dimensional inverse problems, where the unknown
 $u$ belongs to an infinite dimensional function space, is
 less developed.
The ultimate goal of the Bayesian
 approach is to identify a (posterior) probability measure on the
 unknown parameter $u$ using noisy measurements and our prior 
knowledge about $u$. 

One of the first questions that one
might ask is 
whether or not this posterior probability measure is well-defined. While the
answer to this question is relatively straightforward in finite
dimensions, it is far from obvious in the infinite-dimensional setting 
and this is the main focus of this article.

In order to proceed, we introduce some terminology typically associated with \eqref{generci-inverse-problem}. We seek the {\it solution} or {\it parameter } $u$ in (the {\it parameter space}) $X$. The Banach space $X$ may be infinite dimensional, such as the $L^p$
 spaces for $p \ge 1$ or
 the space of continuous
functions. We consider a Borel {\it prior probability measure} $\mu_0$
on $X$. This measure will reflect our prior knowledge of the parameter $u$.
For example, if $u$ belongs to a function space then the prior measure
can dictate whether it is smooth
or merely continuous. The measurement noise
$\eta$ is assumed  $\eta \sim \varrho$ \myhl{(distributed according to $\varrho$)} where $\varrho$ is a Borel probability measure
on $Y$.

 \myhl{Given $u \in X$ we let $\varrho^u$ denote the probability
  measure of $y$ conditioned on $u$. Assuming that $\varrho^u \ll \varrho$
  (i.e. $\varrho^u$  is
  absolutely 
continuous with respect to $\varrho$) and has a density then we can define 
the {\it likelihood potential}
$\Phi(u;y) : X \times Y \to \reals$ so that
$$
  \frac{\dd \varrho^u}{\dd \varrho} (y) = \exp( -\Phi(u;y)), \qquad \int_Y
  \exp(-\Phi(u;y)) \dd\varrho(y) = 1. 
$$
 In finite dimensions $\Phi$ is  
simply the {\it conditional distribution} of
$y$ given $u$ and 
encapsulates our assumptions regarding the distribution of the noise $\eta$ as well
as those on the forward map $\mcl{G}$.
For example, let $Y = \reals^m$ and consider the additive Gaussian measurement noise
model
$$
y = \mcl{G}(u) + \eta, \quad \eta \sim \mcl{N}(0, \pmb{\Gamma}).
$$ 
Here, $\pmb{\Gamma}$ is a $m \times m$ positive 
definite matrix. Then one can use the density of $\eta$ with respect to the Lebesgue measure
in $\reals^m$ to obtain
\begin{equation}
  \label{gaussian-likelihood}
  \Phi(u;y) = \frac{1}{2} \| \pmb{\Gamma}^{-1/2} ( \mcl{G}(u) - y )  \|_2^2.
\end{equation}
} 
 
We can now 
define the {\it posterior probability 
measure} $\mu^y$ (on $X$) via {\it Bayes' rule}
\begin{equation}
  \label{bayes-rule}
  \frac{ \dd \mu^y }{\dd \mu_0} (u) = \frac{1}{Z(y)} \exp \left( -
    \Phi(u;y) \right) \quad  \text{where} \quad Z(y) = \int_X \exp(-\Phi(u; y)) \dd \mu_0(u).
\end{equation}
The Bayesian methodology for inverse problems consists of identifying  the posterior measure
$\mu^y$ \myhl{which is interpreted as an updated version of the
  prior that is informed by the data.} 
The constant $Z(y)$ is simply a normalizing constant that makes $\mu^y$ a 
probability measure on $X$. 
Equation \eqref{bayes-rule} is a generalization of the well-known Bayes' rule 
to general state spaces (see \cite[Sec.~6.6]{stuart-acta-numerica} for an in depth discussion of this generalization). Here, the relationship between the posterior and the prior is understood 
in the sense of the Radon--Nikodym theorem \cite[Thm.~3.2.2]{bogachev1} and so $\mu^y \ll \mu_0$.

Analogous to the situation for partial differential equations (PDEs),
we may ask under what conditions \eqref{bayes-rule} is
uniquely solvable, and whether the posterior probability $\mu^y$
depends continuously on the data. This is loosely what we mean by {\it
  well-posedness} of this problem, though we shall make these
definitions precise shortly. \myhl{We note that our notion of
  well-posedness of the posterior is quite different from the {\it
    consistency of the posterior} introduced by Freedman and Diaconis
  \cite{diaconis-consistency,freedman-consistency}. In their definition,
the posterior measure $\mu^y$ is consistent if it concentrates around the true
value of the unknown parameter as more and more data is collected. In contrast, well-posedness is concerned with the behavior of the
posterior $\mu^y$ when the data $y$ is perturbed and not augmented. 
}

\myhl{
Our broad goal in this paper is to develop a well-posedness theory for Bayesian inverse problems in a large class of priors. Specifically, we  study the well-posedness and consistent approximation of Bayesian
inverse problems with priors that are
{\it convex}. \myhl{Convex measures are also known as 
    ``log-concave'' distributions in the literature. We prefer 
the term ``convex measure'' due to the
connection between convex priors and convex regularization techniques
in variational inverse problems.
 This connection is investigated in detail in the recent articles \cite{burger-map, helin-MAP}
where the authors
study the maximum a posteriori points of Bayesian inverse problems with convex priors.  
}

 The central contribution of our article is the following. If $X$ is a Banach space and $\mu_0$ is a probability measure with certain properties (notably, {\it convexity}), then the Bayesian inverse problem of finding the measure $\mu^y \ll \mu_0$
given by \eqref{bayes-rule} is well-posed under reasonable assumptions on the likelihood, which will be made precise shortly.

 Our results will 
expand the class of prior measures that are available for modelling of prior information in inverse problems.
As we will see later on the class of convex measures already includes the Gaussian and Besov priors 
and so our results will unify some of the existing results in the literature. Furthermore, the class of convex measures includes many of
the priors that are commonly used in the statistics and inverse problems
literature but no theory of well-posedness exists for them,
such as the 
hierarchical priors of \cite{scott-shrink, calvetti-hierarchical}. 
 
}

The questions of well-posedness and consistent approximation of the Bayesian inverse problems have been studied in
\cite{stuart-acta-numerica, cotter-approximation} for
Gaussian priors, in \cite{dashti-besov} for Besov priors and more recently in 
\cite{iglesias-geometric} for geometric priors \myhl{and in \cite{sullivan} for heavy-tailed and stable priors.}  We now present three concrete motivating
examples of inverse problems that use convex prior measures.

% We
%will refer to these examples 
%throughout Sections \ref{sec:well-posedness} and \ref{sec:convex-measures} as we present our theoretical discussions.

\begin{example}[$\ell^1$-regularization of inverse problems]\label{example-deconvolution-ell1}

A popular form of regularization, particularly in the context of
sparse recovery, is $\ell^1$-regularization. 
Let $X = \reals^n$ and $Y = \reals^m$ for fixed integers $m, n >
0$. Suppose that $\mb{A} \in \reals^{m \times n}$ is a fixed matrix
and that the data $y$ is obtained via
  $$y = \mb{A}u + \eta $$ 
where  $u \in
\reals^n$ and $\eta \sim \mcl{N}(0, \sigma^2 \mb{I})$.
Here, $\sigma >0 $ is a fixed constant and $\mb{I} \in
\reals^{m \times m}$ is the identity matrix. Our goal in this problem
is to estimate the parameter $u$ from the data $y$.

If we employ the Bayesian perspective 
then we need to identify the likelihood potential $\Phi$ and the prior
measure $\mu_0$. 
A straightforward calculation yields
\myhl{
\begin{equation}\label{example-1-likelihood}
\Phi: \reals^n \times \reals^m \to \reals, \qquad \Phi(u;y) = \frac{1}{2\sigma^2} \left\| \mb{A} u - y \right\|_2^2.
\end{equation}
}
As for the prior measure $\mu_0$, \myhl{we postulate a model 
\begin{equation}\label{multivariate-laplace}
\frac{\dd \mu_0}{\dd \Theta}(u) = \frac{1}{(2 \lambda)^n}\exp\left( - \frac{\| u\|_1
  }{\lambda} \right)
\end{equation}}
which is a multivariate version of the Laplace distribution (see Table
\ref{tab:convex-distributions-in-1D}). Here, $\Theta$ denotes the
Lebesgue measure on $\reals^n$. \myhl{We can now use Bayes' rule \eqref{bayes-rule} to identify
the posterior measure $\mu^y$ as
$$
\frac{\dd \mu^y}{\dd \Theta}(u)  = \frac{1}{Z(y)} \exp\left( -\frac{1}{2 \sigma^2} \| \mb{A} u -y \|_2^2
    - \frac{1}{\lambda} \| u\|_1  \right).
$$}
Finding the maximizer of this density (i.e. the
maximum a posteriori (MAP) point) corresponds to solving the
optimization problem 
$$
u_{\text{MAP}} := \text{argmin}_{z \in \reals^n}  \frac{1}{2} \| \mb{A} z - y
\|_2^2 + \frac{\sigma^2}{2 \lambda} \| z\|_1.
$$
This is an instance of the well-known $\ell^1$-regularization
technique which is commonly used in recovery of sparse solutions
\cite{foucart}.
\end{example}

The prior measure \eqref{multivariate-laplace} is not Gaussian but, 
as we will see in Section \ref{sec:convex-measures}, it is convex. This example
demonstrates the potential benefits of using non-Gaussian prior
measures in a finite dimensional setting. We now consider a second example that utilizes a
non-Gaussian prior on a function space. 
This example can also be viewed as an infinite dimensional analog of
Example 1.

% The next example

% \NNtodo{Do you want to talk about blind deconvolution and give some references to it? This motivates the need for Laplace priors.}
\begin{example}[Deconvolution]\label{example-deconvolution} 
Let $X = L^2(\mbb{T})$ where $\mbb{T}$ is the circle of radius
$(2\pi)^{-1}$ and let $Y = \reals^m$ for a fixed integer $m$. Let $S : C(\mbb{T}) \to
\reals^m$ be a bounded linear operator that collects point values of a continuous function on a collection of $m$ points over $\mbb{T}$. Finally, given a fixed kernel $g \in C^\infty(\mathbb{T})$,
define
the forward map $\mcl{G}{(\cdot)}:X\rightarrow Y$ as 
\begin{equation}\label{deconvolution-forward-map}
\mcl{G}(u) = S( g \ast u) \qquad \text{where} \qquad (g\ast
u)(x) :=  \int_{\mbb{T}} g(x-y) u(x) dy
. 
\end{equation}
Now suppose that the data $y
$ is obtained via 
$$
y = \mcl{G}(u) + \eta
$$
where $u \in L^2(\mbb{T})$,
$\eta \sim \mcl{N}(0, \sigma^2 \mb{I})$,  $\sigma >0$
is a fixed constant and $\mb{I}$ is the $m \times m$ identity
matrix.
Our goal is to approximate $u \in L^2(\mbb{T})$ from the data $y$. 
Our assumptions imply a likelihood potential \myhl{of a similar form to \eqref{example-1-likelihood}},
\myhl{
\begin{equation}\label{deconvolution-likelihood}
\Phi: L^2(\mbb{T}) \times \reals^m \to \reals, \qquad \Phi(u;y) = \frac{1}{2\sigma^2} \| \mcl{G}(u) - y \|_2^2.  
\end{equation}
}
We now identify the prior measure $\mu_0$ via spectral expansion of its
samples.
 Let $\{ \psi_k \}_{k\in \mathbb{Z}}$ denote the  Fourier basis 
in $L^2(\mathbb{T})$ and
 consider a model of the form
\begin{equation}\label{deconvolution-prior}
\begin{aligned}
 %& \{\xi_k\}_{k\in \mathbb{Z}}, \quad \xi_k \sim {\text{\normalfont Lap}}(0,1), \\
%& \{\gamma_k \}_{k\in \mathbb{Z}}, \quad \gamma_k = (1 + |k|^2)^{-5/4}, \\
& u := \sum_{k\in \mathbb{Z}} \gamma_k \xi_k \psi_k,  \\
\mbox{where} \quad &\xi_k \sim {\text{\normalfont Lap}}(0,1), \quad \gamma_k:=(1 + |k|^2)^{-5/4}, \quad \forall k \in \mathbb{Z}.
\end{aligned}
\end{equation}
The distribution function of the Laplace random variable $\xi$ is
given in Table \ref{tab:convex-distributions-in-1D}.
The prior measure $\mu_0$ will be the probability measure induced by
the random variable $u$. The posterior $\mu^y$ can be
identified via \eqref{bayes-rule}. In order to demonstrate the
potential benefits of using this prior measure 
we shall discretize the problem by truncating the spectral expansion
in \eqref{deconvolution-prior}.

\myhl{
Consider the projection
$$
\Psi_N: L^2(\mbb{T}) \to L^2(\mbb{T}) \qquad 
 \Psi_N (u):= \sum_{k=-N}^{N-1} \langle u, \psi_k \rangle \psi_k,
$$ 
where $\langle \cdot, \cdot \rangle$ denotes the usual $L^2$-inner product. 
We
can approximate the likelihood potential  as 
$$
\Phi_N : L^2(\mbb{T}) \times \reals^m \to \reals, \qquad \Phi_N(u;y) = \frac{1}{2\sigma^2} \| \mcl{G}\left(\Psi_N (u) \right) - y
\|_2^2
$$
and apply Bayes' rule to obtain 
$$
\frac{\dd \mu_N^y}{\dd \mu_0}(u) \propto \exp \left( -\frac{1}{2\sigma^2} \| \mcl{G}\left(\Psi_N (u) \right) - y
\|_2^2 \right
).
$$
Here we think of $\Phi_N$  and $\mu^y_N$ as approximations to the true likelihood
$\Phi$ and posterior $\mu^y$ respectively. 
Now considering the MAP point of $\mu^y_N$ corresponds
to an $\ell^1$-regularized optimization problem similar to Example 1 \cite{lassas-sparsity-promoting}.
}
% Here $\pi^y(\pmb{\xi})$ is simply the Lebesgue density of the
% posterior measure $\mu^y$. Finding the maximizer of this density (i.e. the {\it
%   maximum a posteriori (MAP)} point) corresponds to
% solving the following optimization problem 
% $$
% \pmb{\xi}_{MAP} = \text{argmin}_{\mb{z} \in \reals^{2N}} \quad -\frac{1}{2\sigma^2} \| \mcl{G}(\Psi_N
% (\mb{z})) - y \|_2^2 + \frac{1}{2} \| \mb{z} \|_1
% $$
% which is an instance of the usual $\ell^1$-regularization of inverse
% problems.

\end{example}

The above example is a 
linear inverse problem due to the fact that we assume that the
convolution kernel $g$ is known. If the convolution kernel $g$ is unknown then
this problem is known as the blind-deconvolution problem, which gives rise to a
nonlinear inverse problem. The deconvolution problem
 is a classic ill-posed problem that arises
widely in  optics and imaging, especially in deblurring
applications \cite{vogel, hansen-deblurring}.

The
connection between $\ell^1$-regularization and the Laplace prior is
well-known \cite{somersalo,
  lassas-sparsity-promoting, burger-map} and it serves as motivation
for the study of non-Gaussian prior measures. 
 \myhl{One can prove the
  well-posedness of 
Example 2 using the already established theory of
well-posedness for Besov
priors \cite{lassas-invariant, dashti-besov}. In this article, we
present an 
alternative proof of well-posedness for this problem using a more
general framework. More importantly, our 
well-posedness results will
include more interesting choices of prior measures. We now present an example of such a prior
measure in the context of deconvolution.}

\begin{example}[Deconvolution with a hierarchical prior]\label{deconvolution-hierarchical}
 \myhl{Consider the deconvolution problem of Example
2. Now assume 
the prior samples have the form 
\begin{equation*}
u = \sum_{k \in \mbb{Z}} \gamma_k \zeta_k \xi_k \psi_k
\end{equation*}
where $\gamma_k = (1+ |k|^2)^{-1}$ and the $\psi_k$ are the
Fourier basis functions on $L^2(\mbb{T})$. Let $\{ \zeta_k \}$ and
$\{ \xi_k \}$ be two sequences of i.i.d.\  random variables so that
$\zeta_k \sim \text{Gamma}(2,1)$ and $\xi_k \sim \mcl{N}(0,1)$ (refer
to Table \ref{tab:convex-distributions-in-1D} for the density of the
Gamma random variable). This
construction of the prior can be thought of as a hierarchical prior
model where the modes $\xi_k$ are assumed to be Gaussian random
variables with unknown variances, since $\pi( \xi_k\zeta_k | \zeta_k) = 
\mcl{N}(0, \zeta_k^2)$.
 We can think of the $\zeta_k$ as
a model for the standard deviation of the $\xi_k$.}
\end{example}
  
\myhl{Hierarchical priors
of this form are common in the literature \cite{somersalo,
  calvetti-hierarchical, agapiou, scott-shrink, scott-horseshoe} and
have wide applications. The prior $\mu_0$ that is
induced by the random variable $u$ above is not of the Besov or
Gaussian form and so previous well-posedness results no longer
apply.}

 \myhl{As an immediate application of our theoretical results,  we will be able to
 conclude that the inverse problem in Example 1 and 2 are well-posed. \myhl{Furthermore, we can prove
 the well-posedness of Example 3
by using the fact that Gamma and Gaussian
distributions are convex. We return to the proof of these results in 
Section \ref{sec:examples}.}}

The rest of this article is organized as follows. First we precisely
define the notion of well-posedness and consistent approximation in the next subsection. In Section \ref{sec:well-posedness} we
present a set of general conditions on $\Phi$ and $\mu_0$
which guarantee well-posedness and consistent approximation of
Bayesian inverse problems. Throughout this section we will not assume
that $\mu_0$ is convex.
 In Section \ref{sec:convex-measures}  we collect 
some results about convex probability measures on Banach spaces and show that 
as priors, these measures will result in well-posed inverse
problems. Section \ref{sec:examples} is devoted to 
a general framework for the construction of convex priors on separable Banach spaces 
by means of countable products of one dimensional convex measures. In
Section \ref{sec:examples} we return to consistent approximation of the posterior
measure and we present sharper results concerning the convergence of
the approximate posterior. At the end of this section 
we present four example problems that use convex prior measures.

 % The posterior
% $\mu^y$ is not necessarily well-defined for all choices of the likelihood potential. By the Radon-Nikodym theorem we need $\exp(\Phi(\cdot; y)): X \to \reals$ to be measurable under $\mu_0$ and we also need 
% $0<Z(y) < \infty$.

%In this article we are interested in the notion of {\it well-posedness} of the Bayesian 
%inverse problem. Simply put, we want to identify 
%conditions on $\Phi$ and $\mu_0$ that guarantee the existence and uniqueness of a 
%posterior measure $\mu^y$ which depends continuously on the data $y$.
%% In what follows, $\mcl{M}(X)$ denotes the space of Borel probability measures on $X$ and $d: \mcl{M}(X) \times 
%\mcl{M}(X) \rightarrow \reals$ is a metric on $\mcl{M}(X)$.

\subsection{Key definitions} %{Well-posedness and consistent approximation}
We gather here some key definitions and assumptions. In what follows,
we shall consider the prior probability $\mu_0$ to be in the class of
Radon probability measures on $X$. That is, $\mu_0$ will be an inner
regular probability measure on the Borel sets of $X$, \myhl{(meaning
  that the measure of every set can be approximated by the measure of
  compact sets from within)}. \myhl{We also assume that $\mu_0$ is a 
complete measure i.e., subsets of sets of $\mu_0$-measure zero are measurable. Throughout this article 
the symbol $\nu$ denotes a generic probability measure that is used in the proofs and technical arguments
 and its definition is presented in each context.}
We use the shorthand notation $a \lesssim b$ for $a,b \in \reals^+$
 when there exists a constant $C > 0$ independent of $a,b$ such that $ 0<a \le C
 b$.

\myhl{To make precise the notion of distance between measures,} we shall use the Hellinger metric 
on the space of probability measures on $X$. \myhl{Assuming that $\mu_1$ and $\mu_2$ 
are both absolutely continuous with respect to a third measure $\Lambda$, then the Hellinger distance is defined as
\begin{equation}
  \label{hellinger-metric}
  d_H( \mu_1, \mu_2) := \left( \frac{1}{2} \int_X  \left( \sqrt{ \frac{\dd \mu_1}{ \dd \Lambda}(u) } - \sqrt{ \frac{\dd \mu_2}{ \dd \Lambda}(u)}  \right)^2 \dd \Lambda(u) \right)^{1/2} .
\end{equation}
Alternatively, one can also work with the total variation metric 
\begin{equation}
  d_{TV}( \mu_1, \mu_2) := \frac{1}{2} \int_X  \left|  \frac{\dd \mu_1}{ \dd \Lambda}(u)  -  \frac{\dd \mu_2}{ \dd \Lambda}(u)  \right| \dd \Lambda(u).
\end{equation}
The Hellinger and total variation  metrics are independent of the choice of the measure $\Lambda$ \cite[Lem.~4.7.35]{bogachev1}. 
Furthermore, they impose equivalent topologies due to the following 
set of inequalities
\cite[Thm.~4.7.35]{bogachev1}. 
$$
2 d_{H} (\mu_1, \mu_2) \le d_{TV}(\mu_1, \mu_2) \le \sqrt{8}
d_{H}(\mu_1, \mu_2).
$$
Thus, convergence in one metric implies convergence in the other.}
We note that the Hellinger metric  bounds the difference in the expectations of certain
functions in a particularly simple manner. Suppose that $h: X \to \reals$ is a function so 
that $\int_X h^2(u) \dd \mu_1(u) < \infty$ and $\int_X h^2(u) \dd \mu_2
< \infty$. Then using the Radon--Nikodym theorem and H\"{o}lder's
inequality we have
% \todo{added a discussion about how Hellinger
%   bounds the difference in expectation of functions with bounded variance}
$$
\begin{aligned}
& \left| \int_X h(u)  \right.\left. \dd \mu_1 (u)  - \int_X h(u) \dd \mu_2(u) \right|
 \\ 
& \le \int_X \left| h(u)  \left(\sqrt{\frac{\dd \mu_1}{ \dd \Lambda}(u)}+ \sqrt{\frac{\dd \mu_2}{
    \dd \Lambda}(u)}  \right) \right| \left| \left(\sqrt{\frac{\dd \mu_1}{ \dd \Lambda}(u)} - \sqrt{\frac{\dd \mu_2}{
    \dd \Lambda}(u)}  \right)  \right| 
  \dd \Lambda(u) \\
& \le \left( \int_X \left| h(u)  \left(\sqrt{\frac{\dd \mu_1}{ \dd \Lambda}(u)} + \sqrt{\frac{\dd \mu_2}{
    \dd \Lambda}(u)}  \right) \right|^2 \dd \Lambda(u) \right)^{1/2} 
\left( \int_X  \left| \left(\sqrt{\frac{\dd \mu_1}{ \dd \Lambda}(u)} - \sqrt{\frac{\dd \mu_2}{
    \dd \Lambda}(u)}  \right)  \right|^2 
  \dd \Lambda(u)\right)^{1/2}  \\
& \le {2} \left( \int_X h^2(u) \dd \mu_1(u) + \int_X h^2(u) \dd
\mu_2(u)\right)^{1/2} d_H( \mu_1, \mu_2).
\end{aligned}
$$ 
\begin{definition}[Hellinger Well-posedness] \label{def-existence-uniqueness}
Suppose that $X$ is a Banach space and \myhl{$d_H( \cdot, \cdot)$} is the Hellinger metric on the space of
Borel probability measures on $X$.
Then for a choice of the
prior measure $\mu_0$ and the likelihood potential $\Phi$, the Bayesian inverse problem given by \eqref{bayes-rule}
 is called well-posed if:
  \begin{enumerate}
  \item (Existence and uniqueness) There exists a unique posterior probability measure $\mu^y \ll \mu_0$ given by \eqref{bayes-rule}.
  \item (Stability) Given $\epsilon>0$, there is a constant $ C>0$ such that if $\|y-y'\|_Y <C, d_H( \mu^y , \mu^{y'} ) <\epsilon$.
  \end{enumerate}
\end{definition}
We will also define the notion of {\it consistent approximation} of a Bayesian inverse problem in the context of 
practical applications where one often discretizes the likelihood and 
approximates the posterior by sampling \cite{stuart-mcmc, somersalo}. 
\myhl{Let $\Phi_N: X \times Y \to \reals$ denote an approximation
to $\Phi$ that is parameterized by $N$}, and define an approximation $\mu^y_N$ to 
$\mu^y$ via 
\begin{equation}\label{discretized-bayes-rule}
\frac{\dd \mu^y_N}{\dd \mu_0} = \frac{1}{Z_N(y)} \exp( - \Phi_N(u; y))
\qquad \text{where} \qquad Z_N(y) = \int_{X} \exp(- \Phi_N(u;y) )\dd \mu_0(u). 
\end{equation}

\begin{definition}[Consistent approximation] \label{def-stability}
  The approximate Bayesian inverse problem
  \eqref{discretized-bayes-rule} is a consistent approximation to \eqref{bayes-rule} for a choice of $\mu_0$, 
$\Phi$ and $\Phi_N$ if $d_H( \mu^y, \mu^y_N ) \to 0$ as $| \Phi(u;y) - \Phi_N(u;y) | \to 0$.
\end{definition}

%Therefore, the results of this article can be used to
%prove the well-posedness of a large class of existing problems in the
%literature.

\section{Well-posedness}\label{sec:well-posedness}
In this section we collect certain conditions on the prior measure $\mu_0$ and the likelihood potential 
$\Phi$ that result in well-posed inverse problems. 
\myhl{The results in this section are applicable to exponentially tailed prioirs.}
 We emphasize that the
assumption of convexity of the prior measure is not necessary and will only be considered in Section \ref{sec:convex-measures}. 

Following \cite{dashti-besov} we start with a collection of assumptions on the likelihood potential $\Phi$.
\begin{assumption} \label{assumption-on-likelihood}
Suppose that $X$ and $Y$ are Banach spaces.
Then the function $\Phi: X \times Y \to \reals$ has the
  following properties \label{forward-assumption}:
  \begin{enumerate}[(i)]
  \item (Lower bound in $u$): For an $\alpha_1\ge 0$ and every $r>0$, there is a
    constant $M(\alpha_1,r) \in \reals$ such that $\forall u \in X$
    and $\forall y \in Y$ with $\| y \|_Y < r$, $$\Phi(u;y) \ge M -
    \alpha_1 \| u\|_X. $$

    \item (Boundedness above) For every $r>0$ there is a constant
      $K(r)>0$ such that $\forall u\in X$ and $\forall y \in Y$ with 
$\max \{ \| u \|_X, \|y \|_Y \} < r $, $$ \Phi(u;y) \le K. $$

      \item (Continuity in $u$) For every $r>0$ there exists a 
constant $L(r)>0$ such that $\forall u_1,u_2 \in X$ and $y\in Y$ with 
$\max \{ \|u_1 \|_X, \|u_2 \|_X, \|y \|_Y \} < r$, $$ | \Phi(u_1;y) -
\Phi(u_2,y)| \le L \| u_1 - u_2 \|_X. $$ 

              \item (Continuity in $y$)
For an $\alpha_2 \ge 0$ and for every $r>0$, there is a 
constant $C(\alpha_2, r) \in \reals$ such that $\forall y_1,y_2 \in Y$ with 
$\max \{ \|y_1 \|_Y, \|y_2 \|_Y \} < r$ and $\forall u \in X$, $$ | \Phi(u;y_1) -
\Phi(u,y_2)| \le \exp( \alpha_2 \| u\|_X + C ) \| y_1 - y_2 \|_Y. $$ 
  \end{enumerate}
\end{assumption}
 These four assumptions are
required to ensure the well-posedness of the inverse problem according
to Definition \ref{def-existence-uniqueness}. For example, conditions (i) to (iii) are used
to show the existence and uniqueness of the posterior 
in Theorem \ref{existence-uniqueness}
while (i),(ii) and
(iv) are used to show continuous dependence of the posterior on the
data in Theorem
\ref{robustness}. \myhl{We note that these assumptions are mild and 
hold in many practical applications such as in the case of 
finite dimensional data and additive noise models (see the examples
considered in Section \ref{sec:examples} for details). However, one can still
construct examples where these assumptions no longer hold. We
now present a brief example where Assumption \ref{assumption-on-likelihood}(i)
is not satisfied.} 
\begin{example}(Multiplicative noise model) 
\myhl{  Let $Y = \reals$ and consider the model
$$
y = \eta \mcl{G}(u), \qquad \eta \sim \mcl{U}(-1,1), \qquad
\mcl{G}(u) = \| u\|_X.
$$
Thus, the measurement noise is multiplicative. As before let $\varrho^u$
denote the measure of $y$ conditioned on $u$ and let $\varrho^{u,\eta}$ 
denote the measure of $y$ conditioned on both $u$ and $\eta$. 
Letting ${\delta}$ denote the Dirac delta distribution we can write
$$
\varrho^{u, \eta} (y) = {\delta}( y - \eta \mcl{G}(u)). 
$$
 We obtain the distribution of $y$
conditioned 
on $u$ as
$$
\begin{aligned}
  \frac{\dd \varrho^u}{\dd \Theta}(y) = \frac{1}{2} \int_{\reals} {\delta}(y - \eta
  \mcl{G}(u))
  \mb{1}_{(-1,1)}(\eta) \dd \Theta(\eta) 
 = \frac{1}{2\mcl{G}(u)}  \mb{1}_{(-1,1)}\left( \frac{y}{\mcl{G}(u)}
   \right),
\end{aligned}
$$
where $\Theta$ is the Lebesgue measure as before and the above integral is understood in the sense of
distributions. 
Therefore, the likelihood potential $\Phi$ takes the form 
$$
\Phi(u;y) = \left\{ 
\begin{aligned}
&\ln(\|u\|_X) \qquad &&\|u\|_X \in [0,y) \\
& \infty \qquad && \text{Otherwise}.
\end{aligned}
\right.
$$
Clearly,
 this form of $\Phi$ does not satisfy
Assumption \ref{assumption-on-likelihood}(i) and (ii).
 }
\end{example}

We also impose certain conditions on the prior measure $\mu_0$. Recall that $\mu_0$ by assumption is a Radon measure on $X$, and $\mu_0(X)=1$. 

\begin{assumption} \label{assumption-on-prior}
  The Radon probability measure $\mu_0$ on the Banach space $X$ has
  exponential tails i.e.
    \begin{equation}\label{eq:kappadef} \exists \kappa>0 \quad s.t. \quad  \int_X \exp( \kappa \| u \|_X) \dd \mu_0(u) < \infty\end{equation}% for
     % some $\kappa> \max\{\alpha_1, \alpha_2\}$.
\end{assumption}
The inner regularity assumption on the prior is not very restrictive in practice since 
one often works in separable Banach or Hilbert spaces where
all Borel probability measures are automatically Radon \cite[Thm.~1.2.5]{bogachev-malliavin}.
Assumption \ref{assumption-on-prior} constrains the well-posedness results of this paper to prior measures $\mu_0$ \myhl{that have exponential tails}. We borrow the term ``exponential tails'' from the
theory of probability distributions in finite dimensions. For example
if $\mu_0$ were a measure on $\reals$ with a Lebesgue density $\pi_0(x)$ 
then the condition in
\eqref{eq:kappadef} would reduce to the requirement that $\pi_0(x)$
should decay like an exponential as $|x| \to \infty$.
 %This can also be interpreted as a weaker version of the Fernique Theorem 
%\cite[Thm.~2.8.5]{bogachev-gaussian} for Gaussian measures.
Assumption \ref{assumption-on-prior} is central
to many of the arguments leading to the proof of the well-posedness
and consistent approximation.

We start with the question of existence and uniqueness of the posterior. 
\begin{theorem} \label{existence-uniqueness}
Suppose that $X$ and $Y$ are Banach spaces and 
let the likelihood function $\Phi:X\times Y \to \mathbb{R}$ satisfy
Assumptions \ref{assumption-on-likelihood} (i), (ii) and (iii) with
some constant $\alpha_1$. Also let the prior Radon probability measure $\mu_0$ satisfy 
Assumption \ref{assumption-on-prior}, with a constant $\kappa>0$. If
$\kappa  \ge \alpha_1$ then the posterior $\mu^y$ given by \eqref{bayes-rule} is a well-defined
Radon probability measure on $X$.
\end{theorem}

\begin{proof}
  Our proof will closely follow the method of \cite[Thm.~4.1]{stuart-acta-numerica}. Assumption \ref{assumption-on-likelihood}(iii) ensures the 
continuity of $\Phi$ on $X$ \myhl{and so $\Phi$ is $\mu_0$-measurable. Recall the normalization constant $Z(y):=
\int_X \exp(-\Phi(u; y)) \dd \mu_0(u)$. It remains for us to show that
$0<Z(y) <\infty$ in order to 
conclude that $\mu^y$ is well-defined.} The fact that 
$\mu^y$ is Radon will then follow from the assumption that $\mu_0$ is Radon \cite[Lem.~7.1.11]{bogachev2}.
% Let $E \subset X$ 
%be the continuously embedded Banach space in $X$ as in Theorem \cite{concentration-of-measure}.
To show the boundedness of the normalizing constant, we use Assumption \ref{assumption-on-likelihood}(i) to get
$$
\begin{aligned}
  Z(y) &= \int_X \exp( - \Phi(u;y) ) \dd \mu_0(u)  \\
   & \le \int_X \exp( \alpha_1\| u\|_X - M ) \dd \mu_0(u) = \exp(-M) \int_X \exp(\alpha_1 \| u\|_X) \dd \mu_0(u),
\end{aligned}
$$
which is {bounded when $\alpha_1 \le \kappa$}. 

We now need to show that the normalizing constant $Z(y)$ does not
vanish.
It follows from Assumption \ref{assumption-on-likelihood}(ii) that for each $R > 0$
$$
\begin{aligned}
  Z(y) & = \int_X \exp( - \Phi(u;y)) d\mu_0(u)\\
&   \ge 
 \int_{\{\| u \|_X < R\}} \exp(-K) \dd \mu_0(u) = \exp(-K) \mu_0( \{ \| u \|_X < R \}).
\end{aligned}
$$
\myhl{In order to see that $\mu_0( \{ \| u \|_X < R \} ) > 0$
for large enough $R$, consider the disjoint sets $A_k :=\{ u | k-1\le \| u\|_X <
k\}$ for $k \in \integers$. The $A_k$ are open and hence 
measurable and $\sum_{k=1}^\infty \mu_0(A_k) =
\mu_0(\bigcup_{k=1}^\infty A_k) = \mu(X) = 1$. Then the measure of at least one
of the $A_k$ has to be nonzero. }
% \todo{Check previous sentence}
\end{proof}

\myhl{
Observe that Assumptions \ref{assumption-on-likelihood}(i) and (ii) are needed for the existence and uniqueness of the posterior measure. Therefore, in the case of Example 4 above where both of these assumptions are violated, we 
are unable to prove the existence and uniqueness of the posterior measure. Furthermore, note that in the 
case where the likelihood potential $\Phi$ is bounded from below, i.e. Assumption \ref{assumption-on-likelihood}(i)
is satisfied with $\alpha_1 =0$, we no longer require the prior $\mu_0$ to have exponential tails  for the 
posterior $\mu^y$ to be well-defined. However, we still require the  exponential tails assumption
 in Theorem~\ref{robustness} in order to show the stability of the posterior.  
Finally, 
We note that an alternative proof of the fact 
that $\mu_0( \{ \| u \|_X < R \} ) > 0$
can be obtained by
using the following theorem concerning the concentration of Radon measures on Banach spaces which is interesting in and of itself.

\begin{theorem}[{\cite[Thm.~7.12.4]{bogachev2}}] \label{concentration-of-measure}
Let $\mu$ be a Radon probability measure on a Banach space $X$. Then there exists
a reflexive and separable Banach space $(E, \| \cdot \|_E)$ embedded in $X$ such that $\mu(X \setminus E) =0$ and the closed balls of $E$ are compact in $X$. 
\end{theorem}

The measure $\mu_0$ is supported on the separable space $E$  and 
the closed balls of $E$ are compact in $X$.
 Given $R > 0$ we can
find $R'>0$ so that $\{ \| u\|_E \le R'\} \subseteq \{ \| u\|_X < R\}$. If the measure of centred closed balls of $E$ are zero 
then the measure of all balls of $E$ would have to be zero. Since $E$ is separable,  it can be 
covered by a countable union of balls which would imply $\mu_0(E) = 0$. This  contradicts the fact that $\mu_0$ is concentrated on $E$.}
Also, observe that since $\mu^y \ll \mu_0$ it follows from the definition of absolute continuity that 
$\mu^y(X \setminus E) =0$ as well. Then the posterior $\mu^y$ is also
concentrated on the same separable subspace of 
$X$.

 We now establish the stability of Bayesian inverse problems with respect to 
perturbations in the data.
\begin{theorem}\label{robustness}
  Suppose that $X$ is a Banach space, $\Phi$ satisfies Assumptions
  1(i), (ii) and (iv) with constants $\alpha_1,\alpha_2 \ge0$ and
 let $\mu_0$ 
satisfy Assumption \ref{assumption-on-prior} with a constant $\kappa>0$. Let 
$\mu^y$ and $\mu^{y'}$ be two measures defined via \eqref{bayes-rule} for $y$ and $y' \in Y$,
 both absolutely continuous with respect to 
$\mu_0$. If $\kappa \ge \alpha_1 + 2 \alpha_2$ then  there exists 
a constant $C(r) > 0$ \myhl{such that whenever} $\max\{ \| y \|_Y, \| y' \|_Y \} < r$, 
$$
d_H (\mu^y, \mu^{y'}) \le C \| y - y' \|_Y.
$$
\end{theorem}

\myhl{The result of this theorem can be viewed as a guideline for choosing prior measures in practice. Given a 
likelihood potential $\Phi$ we need to choose a prior measure $\mu_0$ that has sufficiently heavy tails so that 
$\kappa \ge \alpha_1 + 2\alpha_2$ in order to achieve stability. Furthermore, if we already have a prior measure $\mu_0$
that satisfies Assumption \ref{assumption-on-prior} but with a constant $\kappa < \alpha_1 + 2\alpha_2$ then 
we can simply dilate this measure to construct a new measure $\bar{\mu_0} = \mu_0 \circ c^{-1}$ where $c \in (0, \kappa/
(\alpha_1 + 2 \alpha_2))$. This new measure $\bar{\mu}_0$ will satisfy the conditions of Theorem~\ref{robustness} 
and results in a well-posed inverse problem.}
 
A version of Theorem~\ref{robustness} along with Theorem \ref{stability} below is available for Gaussian priors in \cite[Thm.~4.2]{stuart-acta-numerica} and for 
Besov priors in \cite[Thm.~3.3]{dashti-besov}.  
\myhll{Similar results on stability under
perturbation of data and consistent approximation can also be found in the lecture notes \cite[Sec.~4]{stuart-bayesian-lecture-notes}, where the results are established for separable Banach spaces and with slightly different assumptions on the prior
$\mu_0$. Here we present an analog of those proofs under Assumption 2 (exponential integrability), and do not require X to be separable.}

\begin{proof}
Consider the normalizing constants $Z(y)$ and $Z(y')$ associated with $y,y'\in Y$ via \eqref{bayes-rule}. We have already
established in the proof of Theorem \ref{existence-uniqueness} that neither of these 
constants will vanish. Now
applying the mean value theorem to the exponential function and using
Assumptions 1(i), (iv) and
Assumption \ref{assumption-on-prior} with $\kappa > \alpha_1 + 2
\alpha_2$ we obtain
$$
\begin{aligned}
| Z(y) - Z(y') | &\le  \int_X \exp( -\Phi(u;y) ) | \Phi(u;y) - \Phi(u;y') | \dd \mu_0(u)  \\
&\le \left( \int_X \exp( \alpha_1 \| u \|_X - M) \exp( \alpha_2 \| u \|_X + C) \dd \mu_0(u) \right) \| y - y' \|_Y  \\
 &\lesssim \| y -y' \|_Y.
\end{aligned}
$$
On the other hand, following the definition of the Hellinger metric
and using the inequality $(a + b)^2 \le 2( a^2 + b^2)$, we can write
$$
\begin{aligned}
  2d_H^2(\mu^y, \mu^{y'}) &=  \int_X \left( Z(y)^{-1/2} \exp \left( - \frac{1}{2} \Phi(u;y) \right) 
 - Z(y')^{-1/2} \exp \left(- \frac{1}{2} \Phi(u;y')\right) \right)^2 \dd\mu_0(u) \\
& \le \frac{2}{Z(y)} \int_X \left( \exp \left(-\frac{1}{2} \Phi(u;y)\right) - \exp \left(-\frac{1}{2} 
\Phi(u;y') \right) \right)^2 \dd\mu_0(u)  \\
& \quad + 2 \left| Z(y)^{-1/2} - Z(y')^{-1/2} \right|^2 \int_X \exp(- \Phi(u;y') ) \dd\mu_0(u).  \\
 & \quad =: I_1 + I_2.
\end{aligned}
$$
Once again by the mean value theorem and Assumptions 1(iv), (i) and
Assumption \ref{assumption-on-prior} with $\kappa > \alpha_1 +
2\alpha_2$ we have
$$
\begin{aligned}
 \frac{Z(y)}{2} I_1
& \le \int_X \frac{1}{4} \exp(- \Phi(u;y)) |\Phi(u;y') - \Phi(u;y)|^2 \dd\mu_0(u) \\
& \le \int_X \frac{1}{4} \exp(- \Phi(u;y)) \exp( 2\alpha_2 \|u \|_X+2 C) \|y -y'\|_Y^2 \dd\mu_0(u)\\
& \le \int_X \frac{1}{4} \exp(\alpha_1 \| u\|_X - M) \exp( 2\alpha_2 \|u \|_X+2 C)\|y - y'\|_Y^2 \dd\mu_0(u) \\
& \lesssim  \| y - y'\|_Y^2.
\end{aligned}
$$
Furthermore, 
$$
I_2 = 2 |Z(y)^{-1/2} - Z(y')^{-1/2}|^2 Z(y') \lesssim |Z(y) - Z(y')|^2 \lesssim \| y - y'\|^2_Y.
$$
This yields the desired result.
\end{proof}

We now turn our attention to consistent approximation of the inverse problem as per Definition \ref{def-stability}. 
\begin{theorem}\label{stability} Let $X$ and $Y$ be Banach spaces, and
  $\mu_0$ be a prior probability measure on $X$ satisfying Assumption
  \ref{assumption-on-prior} with a constant $\kappa > 0$. 
  Assume that the measures $\mu^y$ and $\mu^y_N$ are defined via \eqref{bayes-rule} and \eqref{discretized-bayes-rule}, for a fixed $y \in Y$, and are absolutely continuous 
with respect to the prior $\mu_0$. Also 
assume that both likelihood potentials $\Phi$ and $\Phi^N$ satisfy
Assumptions 1(i) and (ii) with a constant $\alpha_1 \ge 0$, uniformly for all $N$
and that for an $\alpha_3 \ge 0$ there exists a constant $C(\alpha_3) \in \reals$ so that 
\begin{equation} \label{discretization-assumption}
| \Phi(u;y) - \Phi_N(u;y) | \le \exp(\alpha_3 \| u \|_X + C ) \psi(N)
\end{equation}
where $\psi(N) \to 0$ as $N \to \infty$. If $\kappa \ge \alpha_1 + 2\alpha_3$ then there exists a constant $D$ independent of $N$ so that 
$$
d_H( \mu^y , \mu^y_N) \le D \psi(N).
$$\end{theorem}
\begin{proof}
  The proof of this theorem is very similar to that of Theorem \ref{robustness}. First note that 
$$
\begin{aligned}
| Z(y) - Z_N(y) | &\le  \int_X \exp( -\Phi(u;y) ) | \Phi(u;y) - \Phi_N(u;y) | \dd \mu_0(u)  \\
&\le \left( \int_X \exp( \alpha_1 \| u \|_X - M) \exp( \alpha_3\| u \|_X + C) \dd \mu_0(u) \right) \psi(N)\\
 &\lesssim \psi(N)
\end{aligned}
$$
which follows from applying the mean value theorem followed by
Assumption \ref{assumption-on-likelihood}(i) and
\eqref{discretization-assumption} above as well as Assumption
\ref{assumption-on-prior} with $\kappa \ge \alpha_1 + 2 \alpha_3$. Furthermore, we have
$$
\begin{aligned}
  2d_H^2(\mu^y, \mu^{y}_N) &=  \int_X \left( Z(y)^{-1/2} \exp \left( - \frac{1}{2} \Phi(u;y) \right) 
 - Z_N(y)^{-1/2} \exp \left(- \frac{1}{2} \Phi_N(u;y)\right) \right)^2 \dd\mu_0(u) \\
& \le \frac{2}{Z(y)} \int_X \left( \exp \left(-\frac{1}{2} \Phi(u;y)\right) - \exp \left(-\frac{1}{2} 
\Phi_N(u;y) \right) \right)^2 \dd\mu_0(u)  \\
& \quad + 2 \left| Z(y)^{-1/2} - Z_N(y)^{-1/2} \right|^2 \int_X \exp(- \Phi_N(u;y) ) \dd\mu_0(u).  \\
 & \quad =: I_1 + I_2.
\end{aligned}
$$
It then follows in a similar manner to proof of Theorem \ref{robustness}
 that $I_1 \lesssim \psi(N)$ 
and $I_2 \lesssim \psi(N)$ if $\kappa > \alpha_1 + 2 \alpha_3$
 which gives the desired result.
\end{proof}

In order to provide more details about the
rate of convergence of $\mu^y_N$ to $\mu^y$, we need to  impose further
assumptions on $\Phi$ such as stronger continuity assumptions or the
behavior of the function $\psi(N)$.
We will return to this question in
Section \ref{sec:examples} and provide finer results in the case where $\Phi_N$ is
obtained from  projections of $u$ onto finite dimensional
subspaces of $X$. \myhl{We emphasize that the result of Theorem \ref{stability} holds for a generic approximation $\Phi_N$ of the likelihood potential
$\Phi$ and is not restricted to cases where $\Phi_N$ is obtained via discretization. For example \cite{stuart-GP}
studies consistent approximations of the posterior measure by Gaussian process emulators.} 

\section{Convex measures}\label{sec:convex-measures}
Up to this point, we have shown that Assumptions \ref{assumption-on-likelihood} and \ref{assumption-on-prior} are sufficient for establishing the well-posedness of
Bayesian inverse problems for a broad class of priors. Assumptions \ref{assumption-on-likelihood} 
are properties of the model for the measurements  and  do not depend on the prior. Since the focus of this article is on the prior measure 
we dedicate this section to showing that Assumption \ref{assumption-on-prior} holds for a
large class of priors that go beyond the
 Gaussians \cite{stuart-acta-numerica} and the Besov priors
\cite{dashti-besov}. We start by collecting some results on
the class of convex probability measures on Banach spaces. Our main
reference 
is \cite{borell-convex} where the theory of convex probability
measures on topological vector spaces is developed. In what follows
$\mcl{B}(X)$ denotes the Borel $\sigma$-algebra on a Banach space $X$.

\begin{definition}\label{def:convex-measure}
  Let $\mu$ be a Radon measure on a Banach space $X$. We say that $\mu$ belongs to the class
of convex measures on $X$ if for all $0\le \lambda\le 1$ and sets $A,B \in \mcl{B}(X)$ we have
  \begin{equation}
\label{convex-inequality}
    \mu (\lambda A + (1-\lambda)B) \ge \mu(A)^\lambda \mu(B)^{1-\lambda}.
  \end{equation}
\end{definition}
Equivalently, convex measures on Banach spaces can be identified by
their finite dimensional projections.
\begin{theorem}[{\cite[Thm.~2.1]{borell-convex}}] \label{convex-finite-dim-projection}
  A Radon probability measure $\mu$ on a Banach space $X$ is convex precisely when 
$\mu_{\ell_1, \cdots, \ell_n}$ is a convex measure on $\reals^n$ for
all integers $n$ and elements $\ell_i \in X'$ the  dual of $X$ where 
$$
\mu_{\ell_1,\cdots,\ell_n}(A)  := \mu( \{ B \in \mathcal{B}(X): (\ell_1(B), \cdots, \ell_n(B)) \in A \} ) \: \text{ for all sets } A \in \mcl{B}(\reals^n).
$$
\end{theorem}
In the case of $\reals^n$, convex measures are easily identified by their 
Lebesgue density.
\begin{theorem}({\cite[Thm.~1.1]{borell-convex}})\label{convex-measure-density}
  A Radon probability measure $\mu$ on $\reals^n$ is convex precisely when there exists 
an integer $0 < k \le n$, a Radon probability measure $\nu$ on $\reals^k$ with 
Lebesgue density $\pi_\nu$ and an affine mapping $h$ such that $\mu = \nu \circ h^{-1}$ and 
$- \log(\pi_\nu)$ is convex.
\end{theorem}

The above theorem allows us to easily identify or construct convex
measures in finite dimensions.
We summarize a list
of common convex measures on $\reals$ in Table
\ref{tab:convex-distributions-in-1D}. \myhl{Note that some of these
  measures such as the exponential, Gamma and uniform distribution
are supported on subsets of $\reals$. This fact does not affect 
the results of this section}.  The proof of convexity for these
distributions follows
from showing that their Lebesgue densities are log-concave (see \cite{bagnoli-log-concave}
for details).
%   \begin{itemize}
%   \item Gaussian ($\mcl{N}(m, \sigma^2$)): $\pi(x) = \frac{1}{\sqrt{2\pi \sigma^2}} \exp\left( -\frac{(x-m)^2}{2\sigma^2}\right)$ for $m \in \reals$ and $\sigma >0$.
% \item  Exponential: $\pi(x) = \lambda \exp( - \lambda x), \lambda > 0$ and supported on $[0, \infty)$.
% \item Logistic: $\pi(x) = \frac{\exp\left(- \frac{x-m}{s}\right)}{ s \left( 1 +  \exp\left(- \frac{x-m}{s}\right) \right)^2} $ with $m \in \reals$ and $s > 0$.
% \item Laplace: $\pi(x) = \frac{1}{2\lambda} \exp( -|x - m| / \lambda)$ for 
% $m \in \reals$ and $\lambda > 0$.
% \item  $\text{Gamma}(k, \lambda)$: $\pi(x) =\frac{1}{\Gamma(k) \lambda^k} x^{k-1} \exp(-x/\lambda)$ for $k \ge 1$, $\lambda > 0$ and $x \in (0, \infty)$.
% \item Inverse Gamma: $\pi(x) =\frac{\lambda^k}{\Gamma(k) } x^{-k-1} \exp(-\lambda/x)$ for $k ,\lambda > 0$ and $x \in (0, \infty)$. 
% \item Uniform distribution on any convex subset of $\reals$.
%   \end{itemize} 
  \begin{table}[htp]
    \centering
    \begin{tabular}{| l | c | c | c |}
      \hline 
      Distribution & Symbol & Lebesgue density $\pi(x)$ & Parameter
                                                          Range \\
      \hline
      Gaussian & $\mcl{N}(m, \sigma^2)$ & $\frac{1}{\sqrt{2\pi
                                         \sigma^2}} \exp\left(
                                         -\frac{(x-m)^2}{2\sigma^2}\right)$
                                                        & $m \in
                                                          \reals$ and
                                                          $\sigma >0$
      \\ \hline
Exponential & $\text{Exp}(\lambda)$ & $\mb{1}_{[0, \infty)}(x)\lambda \exp( - \lambda x)$&
                                                                    $\lambda
                                                                    >
                                                                    0$
      \\ \hline
Laplace & $\text{Lap}(m, \sigma)$ & $\frac{1}{2\sigma} \exp\left( -
                                    \frac{|x - m |}{\sigma} \right) $
                                                        & $m \in
                                                          \reals $ and
                                                          $\sigma > 0$ 
      \\ \hline
Logistic & $\text{Logistic}(m, s)$ &$\frac{\exp\left(-
                                     \frac{x-m}{s}\right)}{ s \left( 1
                                     +  \exp\left(-
                                     \frac{x-m}{s}\right) \right)^2} $
                                     & $m \in \reals$ and $s > 0$ \\
      \hline
Gamma & $\text{Gamma}(k, \lambda)$&  $\mb{1}_{[0, \infty)}(x)\frac{1}{\Gamma(k) \lambda^k}
                                    x^{k-1} \exp(-x/\lambda)$&
 $k \ge 1$ and $\lambda > 0$ \\ \hline
%Inverse Gamma & $\text{inv-Gamma}(k, \lambda)$ &
  %                                               $\frac{\lambda^k}{\Gamma(k)
      %                                           } x^{-k-1}
    %                                             \exp(-\lambda/x)$ &
        %                                                             $k
              %                                                       ,\lambda
            %                                                         >
          %                                                           0$
      %\\ \hline
Uniform & $\mcl{U}(a,b)$ & $\mb{1}_{[a,b]}(x) $ & $ a, b \in \reals $ and
                                               $b > a$ \\ \hline
    \end{tabular}
\caption{List of common distributions on $\reals$ that are convex
  within the prescribed parameter range.} \label{tab:convex-distributions-in-1D}
  \end{table} 

The following corollary to Theorem~\ref{convex-finite-dim-projection}  is the cornerstone of our recipe for construction of convex 
prior measures in \myhl{Section \ref{construction-of-convex-priors}}.
\begin{corollary}\label{convexity-preservation}
  \begin{enumerate}[(i)]
\item The image of a convex measure under a continuous linear mapping is convex. 
\item Let $X$ be a Banach space and let $\{ \mu_n \}_{n=1}^\infty$ be a sequence of convex measures on $X$
converging weakly to a measure $\mu$. Then $\mu$ is also convex.
\item The product or convolution of two convex measures is also convex.

% \item Let $\{ X_n \}_{n=1}^\infty$ and $\{ \mu_n \}_{n=1}^\infty$ be a countable sequence of Banach spaces and
% corresponding convex measures. Then the measure $\mu = \bigotimes_{n=1}^\infty \mu_n$ 
% is a convex measure on the Banach space $X = \bigotimes_{n=1}^\infty X_n$.
  \end{enumerate}
\end{corollary}
\begin{proof}
  Parts (i) and (ii) are proven in
  \cite[Lem.~2.1,Thm.~2.2]{borell-convex} and so we only provide the details for
  (iii). Suppose that  $\mu$ and \myhl{$\rho$} are convex measures on Banach
  spaces $X$ and $Z$ and let \myhl{$\mu \otimes \rho$} denote the product
  measure on $X \times Z$. Also, let $\mathcal{B}(X) \otimes
  \mathcal{B}(Z)$ denote the product Borel $\sigma$-algebra on $X$ and $Z$.
For every set $D \in \mcl{B}(X) \otimes \mcl{B}(Z)$ and given $z\in Z$, define the slice
$D_z := \{ x \in X | (x,z) \in
  D\}$. For $A,B \in \mcl{B}(X) \otimes \mcl{B}(Z)$ and a constant 
$0 \le\lambda  \le1$ we have 

$$
(\lambda A + (1- \lambda)B)_z \supseteq \lambda A_z + (1- \lambda) B_z.
$$

Then, by Fubini's theorem \cite[Theorem~3.4.1]{bogachev2} we have 
\myhl{$$
\begin{aligned}
\mu \otimes \rho (\lambda A + (1-\lambda) B) &= \int_Z \mu( (\lambda A +
(1- \lambda) B)_z) \rho( dz) \\ &\ge \int_Z \mu( \lambda A_z +
(1- \lambda) B_z) \rho( dz) \\ & \ge  
\int_Z \mu(A_z)^\lambda
\mu(B_z)^{(1-\lambda)} \rho(dz)  \ge
\left( \int_Z \mu(A_z) \rho(dz) \right) ^\lambda
\left( \int_Z \mu(B_z) \rho(dz) \right)^{(1-\lambda)}
\end{aligned}
$$}
where the last two inequalities follow from the convexity of $\mu$ and
H\"{o}lder's inequality. The assertion on the convexity of the
convolution follows by considering the convolution  \myhl{$\mu \ast
\rho$ as the image of $\mu \otimes \rho$ under the linear mapping $(x,z)
\mapsto x + z$.}
% As for
% assertion (iv), let $P_k: X \to \bigotimes_{n=1}^k X_n$ denote the
% natural projections of $X$ down to $\bigotimes_{k=1}^n
% X_n$. Furthermore, 
% let $e_k: X \to X_k$ denote the usual coordinate functionals
% i.e. for any $u = (u_1, u_2, \cdots) \in X$, $e_k(u) = u_k$. Now
% observe that any bounded linear functional on $X$ consists of
% a finite linear combination of the $e_k$ composed with appropriate elements of
% $X_k'$. Therefore, given $\ell_1, \cdots, \ell_m \in X'$ there exists
% a number $q$ and a linear mapping $h$ so that 
% $$
% \mu_{\ell_1, \cdots, \ell_m} = \mu_{e_1, \cdots, e_q} \circ h^{-1}  =
% \left(\mu \circ
% P_q^{-1}\right) \circ h^{-1}.
% $$ 
% Then it follows from assertions (i) and (ii) that all finite dimensional
% projections of $\mu$ are convex measures and so by Theorem
% \ref{convex-finite-dim-projection} $\mu$ is also convex.
\end{proof}

We also recall the following result which identifies the condition for convexity of measures that are 
absolutely continuous with respect to another convex measure. 
\begin{theorem}[{\cite[Prop.~4.3.8]{bogachev-malliavin}}]  \label{absolutely-convex}
    Let $\mu$ be a convex probability measure on a Banach space $X$ and let $V$ denote a 
continuous, measurable and convex function such that $\exp(- V)$
is $\mu$-measurable. 
\myhl{If $0<Z = \int_X \exp(-V(u)) \dd\mu(u)<\infty$ 
then 
the measure $\nu$ defined via 
 $\frac{\dd\nu}{\dd \mu} = \frac{1}{Z} \exp(-V)$ is a convex probability measure.}
\end{theorem}

The most attractive feature of convex measures in the context of inverse problems is 
that, as priors, they satisfy Assumption \ref{assumption-on-prior}. 
First we need a technical result concerning convex measures on Banach
spaces. 

\begin{theorem}[{\cite[Thm.~4.3.7]{bogachev-malliavin}}]
\label{bogachev-convex-measurable-exp}
\myhl{  Let $\mu$ be a convex measure on a Banach space $X$ and let 
$q$ be a $\mu$-measurable seminorm which is finite 
$\mu$-a.e. Then $\int_X\exp( c q(u)) \dd \mu(u) < \infty$ provided that 
$\mu( \{ u| q(u) \le c^{-1}\}) > \frac{e^2}{1 + e^2}$.}
\end{theorem}

%\NNtodo{Could you take a look at the proof below?}
\begin{theorem} \label{convex-measurable-exp}
  Let $X$ be a Banach space and let $\mu$ be a convex probability
  measure so that $\mu(X) =1$ and $\| u\|_X < \infty$ $\mu$-a.s. 
Then there exists a $\epsilon_0>0$ such that $\mu( \{ \epsilon_0\| u \|_X
\le1 \} ) > \frac{e^2}{1 + e^2} $. Moreover, $\int_X \exp( \epsilon \| u \|_X) \dd \mu(u) < \infty$ provided $0\leq \epsilon \leq \epsilon_0$.
%\todo{added this condition from Bogachev.}
\end{theorem}
\begin{proof}
Since $\mu$ is a convex measure then it is Radon by definition and so
there exists a compact set $K \subset X$ such that $\mu(K) >
\frac{e^2}{1 + e^2}$. Since $K$ is compact then it can be covered by a
finite collection of open balls which in turn implies that it can be
covered by a single open ball in $X$. This guarantees the existence of the
constant $\epsilon_0 >0$ so that $\mu( \{ \epsilon_0\| u\|_X \le 1 \}) >
\frac{e^2}{1 + e^2}$. 
% by Theorem~\ref{concentration-of-measure} it is supported on a
% separable and reflexive Banach space $E \subseteq X$ and closed balls of
% $E$ are compact in $X$. Furthermore, there exists a constant $c >0$ so that 
% $\mu(\{ c \| u\|_E \ \le 1 \}) > \frac{e^2}{1 + e^2} $. The existence
% of such a constant is guaranteed by the fact that $\mu(\{ \| u\|_E \le
% R\}) >0$ for a large enough constant $R>0$ according to our argument
% in the proof of
% Theorem~\ref{existence-uniqueness}. Now the existence of the constant
% $\epsilon_0> 0$ follows because we can pick $\epsilon_0 < c$ so that 
% $\{ c\| u\|_E \le  1\} \subseteq \{ \epsilon_0 \|u\|_X < 1\}$.
\myhl{Then from the convexity of $\mu_0$ and
 Theorem~\ref{bogachev-convex-measurable-exp} it follows that
$\exp( \epsilon_0 \| u\|_X )$ is integrable.}
\end{proof}

% \begin{proof}
% It follows from Theorem \ref{concentration-of-measure}
% that $\mu$ is concentrated on a linear separable subspace $E \subset X$. The
% balls of $E$ are compact in $X$
% and so $\| \cdot \|_E$ is $\mu$-measurable. 
% \NNtodo{You start with a Banach space X, and have a probability measure on it. It also has a norm, $\| \cdot\|_X$. Why is it that $\| u\|_X <\infty$? This is what you seem to be implying below.}
% Further, $\| \cdot
% \|_E$ is bounded a.s. To see this, suppose that $\| \cdot \|_E$ is not 
% bounded a.s. This implies that there is a set $A \subseteq E$ so that
% $\mu(A) > 0$ and $\| \cdot \|_E$ is unbounded on $A$. This clearly
% contradicts the fact that $E \subseteq X$. 

%  Then it follows from \cite[Thm.~4.3.7]{bogachev-malliavin}
% that $\int_X \exp( \epsilon \| u \|_E ) \dd \mu(u) < \infty$ for a
% small enough $\epsilon >0$ such that $\mu( \{ \epsilon\| u \|_X
% \le1 \} ) > \frac{e^2}{1 + e^2} $. Now the claim follows since we can
% simply take $\| \cdot \|_E = \| \cdot \|_X$. \NNtodo{This is really REALLY not clear. If it is indeed true that you can take the same norms, why not work with them from the start? You are claiming, essentially, that $\|\cdot\|_X$ is $\mu-$ measurable and finite a.s.}
% \end{proof} 
\myhl{We emphasize that the key step in the proof of Theorem
  \ref{bogachev-convex-measurable-exp} is the assumption that the
  measure 
$\mu$ is convex. The proof of this theorem is rather technical \cite[pp.~114]{bogachev-malliavin}. To gain some intuition about this result consider a convex random variable $\xi$
in $\reals$ with Lebesgue density $\pi_{\xi}(x) = \exp(- \mcl{R}(x))$,
where $\mcl{R}$ is a convex function. The tail behavior of $\pi_\xi$
is tied to the rate of growth of $\mcl{R}$. The
slower the growth of $\mcl{R}$ is the heavier the tails of $\pi_\xi$
will be. However, if $\mcl{R}$ is convex then at the very least $\mcl{R}(x) =
\mcl{O}(|x|)$ 
for large $|x|$. This means that the tails of $\pi_\xi$ will decay at
least like that of the Laplace distribution (with Lebesgue density $\frac{1}{2\lambda} \exp(-
\lambda |x|)$), which also has exponential
tails. Therefore, $\xi$ has exponential tails as well.}  

We are now in a position to connect the theory of convex measures to the question of well-posedness
of inverse problems from the previous section. 
\begin{theorem}\label{convex-bayesian-inverseproblems}
Let  $X, Y$ be Banach spaces and let $\mu_0$ be a convex probability
measure on $X$. Let $\kappa >0$ be the
largest constant so that $\int_X \exp( \kappa \| u\|_X) \dd \mu_0(u) < \infty$,
following Theorem \ref{convex-measurable-exp}. Let $\Phi:X\times
Y\rightarrow \mathbb{R}$ be a likelihood potential satisfying
Assumption \ref{assumption-on-likelihood} with constants $\alpha_1,
\alpha_2 \ge 0$. 
Given a fixed $y \in Y$, consider the Bayesian inverse problem of finding the measure $\mu^y \ll \mu_0$
given by \eqref{bayes-rule}.
 Then if $\kappa \ge \alpha_1 +2  \alpha_2 $
\begin{enumerate}[(i)]
\item The inverse problem \eqref{bayes-rule} is well-posed.

\item If $\Phi_N$  satisfies the conditions of Theorem \ref{stability}
  with $\alpha_3\ge 0$ and $\kappa \ge \alpha_1 + 2 \alpha_3$
then the resulting approximation $\mu^y_N$ to the posterior $\mu^y$ is consistent.
\item If $\Phi(\cdot; y) : X \to \reals$ is a convex function then
  $\mu^y$ is a convex measure.
\end{enumerate}
 \end{theorem}
 \begin{proof}
The assertions follow from Theorems \ref{existence-uniqueness}, \ref{robustness}, \ref{stability}, \ref{absolutely-convex}.
 \end{proof} 

A few remarks are in order regarding the condition on the constant
$\kappa$.
 First, Theorem
\ref{convex-measurable-exp} guarantees that a constant $\kappa >0$
exists as long as $\mu_0$ is a convex measure. As before, given the
constants 
$\alpha_1$ and $\alpha_2$,  if we start with
a convex measure ${\mu}_0$ for which $\kappa < \alpha_1 + 2 \alpha_2$ then we can simply dilate this measure to obtain a new
  measure 
  $\bar{\mu}_0  =  \mu_0\circ c^{-1}$ where $c \in (0, \kappa/(
  \alpha_1 + 2 \alpha_2)) $. This new
  measure 
will satisfy the conditions of Theorem
\ref{convex-bayesian-inverseproblems}.
Finally, we have the following corollary to this theorem which 
gives
the 
well-posedness of the inverse problem without the need to specify the
parameter $\kappa$ whenever the likelihood $\Phi$ satisfies a 
stronger version of Assumption \ref{assumption-on-likelihood}(i) and (iv).
\begin{corollary}\label{convex-bayesian-inverseproblems-poly-bounded}
Consider the setting of Theorem
\ref{convex-bayesian-inverseproblems}. If $\Phi$ satisfies Assumption 
\ref{assumption-on-likelihood}(i) and (iv) for any constants
$\alpha_1, \alpha_2 > 0$ then the result of Theorem
\ref{convex-bayesian-inverseproblems} holds if $\mu_0$ is any convex
measure such that $\mu_0(X) = 1$.
\end{corollary}

\subsection{Construction of convex priors} \label{construction-of-convex-priors}
%\NNtodo{Do you want to add a reference and paragraph around the recent work of Marzouk here?}
So far we have established that convex measures are a convenient
choice as priors for inverse problems mainly due to the fact that they
have exponentially decaying tails.
However, we need to 
discussed a method for construction of 
convex prior measures in practical situations. It is often difficult
to find a convex measure with a 
simple form, such as a Gaussian that can be identified by a mean function and a covariance operator. 
 In this section we consider a general recipe for the construction of convex 
priors on Banach spaces that have an
unconditional basis. This is relevant to practical applications since
interesting Banach or Hilbert spaces are often separable and have a
basis. Interesting
spaces include the $L^p$ spaces, Sobolev and Besov
spaces or the space of continuous functions on bounded intervals.

In the following we use the shorthand notation $\{ c_k \} $ 
to denote a sequence of elements $\{c_k\}_{k=1}^\infty$. Also, $\EE \xi$
and $\VV \xi$ denote the expectation and variance of the random
variable $\xi$.
 We will need an assumption regarding the space $X$, which will greatly simplify the construction: 
%Suppose that $X$ is a Banach space and it satisfies the following assumption:
\begin{assumption}\label{assumption-on-X}
 $X$ is a Banach space with an unconditional
  and normalized Schauder basis $\{ x_k\}$.
\end{assumption}

% These conditions then imply that $X$ is isomorphic to a sequence space.
% \begin{theorem}[{\cite[Thm.~4.15]{heil-basis}}]
% The space $X$ is isomorphic to the sequence space 
% \begin{equation}
%   \label{sequence-space}
%   W := \left\{ \{w \} : \{ w \} \in \reals \text{ and } \sum w_n x_n  \text{ converges in } X \right\}, \quad  \| \{w\} \|_W = \sup_{n} \left\| \sum_{k=1}^n w_k x_k \right\|_X.
% \end{equation}
% $W$ is a Banach space which is topologically isomorphic to $X$ 
% via the synthesis operator $T: \{w\} \mapsto \sum w_k x_k$. Furthermore, the following statements hold:
% \begin{enumerate}[(i)]
% \item The partial sum operators $S_n(x) := \sum_{k=1}^n a_k(x) x_k$ are bounded 
% and $\| S_n \| \le \| T^{-1} \|$ for all $n \in \mathbb{N}$. 
% \item Let $ \beta = \sup_n \| S_n \| < \infty $. Then $\| x \|_{XS} := \sup_n \| S_n(x)\|_X$ is an equivalent norm on $X$ and $\| \cdot \|_X \le \| \cdot \|_{XS} \le \beta \| \cdot \|_X$.
% \item The coefficient functionals $a_k \in X^\ast$ satisfy
% $1 \le \| a_k \|_{X^\ast} \| x_k \|_X \le 2\beta$ for all $k \in \mathbb{N}$.
% \end{enumerate}
% \end{theorem}
% We shall use the natural product structure of $W$ to construct a
% convex prior measure on $X$. 

Our construction of the convex priors is 
inspired by the Karhunen--Lo{\'e}ve expansion of Gaussian measures \cite[Thm.~3.5.1]{bogachev-gaussian} and the construction of the Besov priors in \cite{lassas-invariant, dashti-besov}.
Consider a sequence of convex probability measures 
$\{ \mu_k\}$ on $\reals$ and
corresponding random variables $\xi_k \sim \mu_k$ as well as a fixed
sequence 
$\{ \gamma_k \}$. Next, define the random variable
\begin{equation}
  \label{product-prior-sample}
  u  = \sum_{k =1}^\infty \gamma_k \xi_k x_k
\end{equation}
and take the prior measure $\mu_0$ to be the probability measure that is
induced by this random variable. \myhl{The prior measures in Examples \ref{example-deconvolution} and \ref{deconvolution-hierarchical} in Section \ref{sec:introduction}
are concrete examples of this type of prior.} Certain conditions on
$\{ \xi_k \}$ and $\{ \gamma_k \}$ are needed to ensure that the
induced measure $\mu_0$
is well-defined and is supported on $X$.

\begin{theorem}\label{product-prior-is-convex-ellp}
Suppose that $X$ satisfies Assumption \ref{assumption-on-X}.
Let $u$ be defined as in \eqref{product-prior-sample} and let $\mu$ be
the probability measure induced by $u$ on $X$. If $\{ \gamma_k^2 \} \in \ell^p(\reals)$ and $\{ \VV
  |\xi_k| \} \in \ell^q(\reals)$ for $1\le p, q \le \infty$ so that
  $1/p + 1/q = 1$ (with $p=1$ in the limiting case when $q=\infty$)
  then
  \begin{enumerate}[(i)]
  \item $\| u \|_X < \infty$ a.s.
      \item $\mu$ is a convex measure on $X$.
    \item $\int_X \exp( \epsilon \| u \|_X ) < \infty$ for a small
      enough $\epsilon >0$.
  \end{enumerate}
\end{theorem}

\begin{proof}
(i) By H{\"{o}}lder's inequality we have
$$
\sum_{k=1}^\infty \VV |\gamma_k \xi_k| = \sum_{k=1}^\infty |\gamma_k|^2
\VV |\xi_k| \le \| \{ |\gamma_k|^2 \} \|_{\ell^p} \| \{ \VV |\xi_k| \}
\|_{\ell^q} < \infty.
$$
Therefore, by Kolmogorov's two series theorem
\cite[Lemma~3.16]{kallenberg} we have that $\sum_{k=1}^\infty
|\gamma_k \xi_k | <\infty$ a.s.
Now consider positive integers $M > N > 0$ and let $u_N = \sum_{k=1}^N
\gamma_k \xi_k x_k$. Then,
using the triangle inequality we have 
$$
\| u_M - u_N \|_X = \left\| \sum_{k=N+1}^M \gamma_k \xi_k x_k \right\|_X
\le \sum_{k=N+1}^\infty |\gamma_k \xi_k |.
$$
Taking the limit as $M,N \to \infty$ the sum on the right hand
side will vanish a.s. Then, by the dominated convergence theorem
 the sequence $\{ u_N\}_{N=1}^\infty$ is
Cauchy a.s. and so $\| u \|_X < \infty$ a.s.

(ii) Let $K_i$ denote the support of $\mu_i$ on $\reals$ and let
$\nu_i$ denote the sequence of  probability measures that are obtained by
restricting $\mu_i$ to $K_i$.
Let $\tilde{\nu} := \bigotimes_{i =1}^\infty \nu_i$ denote the quasi-measure that is 
generated by the countable product of the $\nu_i$ on $K:=\bigotimes_{i=1}^\infty K_i$.
By \cite[Thm.~7.6.2]{bogachev2} $\tilde{\nu}$ has a unique extension to a Radon measure $\nu$
on $K$. Now consider the operator
$$
Q : K \to X \qquad Q( \{ c_k \}) = \sum_{k=1}^\infty \gamma_k c_k x_k. 
$$
This operator is well defined following (i), and it is linear
and continuous. Take  $\mu$ to be the push-forward of the measure $\nu$ under the mapping 
$ Q$. Then it follows from Corollary
\ref{convexity-preservation} that
$\mu$ is a convex measure on $X$.

% Pick any $f \in X^\ast$ (the dual of $X$). Since $\| u \|_X <
% \infty$ a.s then $f(u) \le \| f\|_{X^\ast} \| u \|_X < \infty$ a.s.
% Then 
% $\mu\circ f^{-1}$ is a convex measure on $\reals$ following Corollary
% \ref{convexity-preservation}(i) which implies that all finite
% dimensional projections of $\mu$ are convex and so $\mu$ is a convex
% measure on $X$
% by Theorem \ref{convex-finite-dim-projection}. 
(iii) This result
follows directly from the previous assertions and Theorem \ref{convex-measurable-exp}.
%  Then, $f(u) = \sum_{i=1}^\infty \gamma_i \xi_i 
% f(x_i) \le \| f \|_{X^\ast}\sum_{i=1}^\infty | \gamma_i \xi_i| 
% $. Once again the sum on the right hand side is bounded almost surely
% by Lemma \ref{sum-of-positive-random-variables}. Then the measure
% $\mu\circ f^{-1}$ is convex following Corollary
% \ref{convexity-preservation}(i). This implies that all finite
% dimensional projections of $\mu$ are convex and so $\mu$ is also a
% convex measure by Theorem \ref{convex-finite-dim-projection}. Finally, assertion {\it (iii)} follows directly from 
\end{proof} 
\begin{corollary} \label{product-prior-is-convex-iid}
Let $u$ be defined as in \eqref{product-prior-sample} and assume that
the $\{ \xi_k\}_{k=1}^\infty$ are i.i.d.\  convex random variables that are centered. If $\{
\gamma_k^2 \} \in \ell^1(\reals)$ then the result
of Theorem \ref{product-prior-is-convex} holds.
\end{corollary}
\begin{proof}
  By Theorem \ref{convex-measurable-exp} we have that $\VV\xi_1^2 <\infty$ and so $\{ \VV \xi_k \} \in \ell^\infty(\reals)$. The
  result now follows from Theorem \ref{product-prior-is-convex-ellp}.
\end{proof}

It is possible to prove different versions of Theorem
\ref{product-prior-is-convex} with other conditions on the 
sequences $\{ \gamma_k \}$ and $\{ \xi_k \}$. We shall consider 
one more version of this theorem that
may be of interest in practical applications.
 First, we recall the following
classical result on the convergence of a sequence of positive random
numbers. 
\begin{lemma}[{\cite[Proposition~3.14]{kallenberg}}] \label{sum-of-positive-random-variables}
  Let  $ \{ \xi_k\}$ be a sequence of random variables on
  $\reals^+$. Then $\sum_{k=1}^\infty \xi_k < \infty$ a.s. if and only
  if $\sum_{i=1}^\infty \EE \min \left(  \xi_k, 1\right) <\infty$.
\end{lemma}

\begin{theorem} \label{product-prior-is-convex}
Let $u$ be defined as in \eqref{product-prior-sample} and let $\mu$ be
the probability measure induced by $u$ on $X$. Then the result of
Theorem \ref{product-prior-is-convex-ellp} holds if the assumptions on
$\{ \gamma_k\}$ and $\{\xi_k\}$ are replaced by the assumption that $\sum_{i=1}^\infty
\EE  \min \left( | \gamma_k \xi_k |,  1 \right)  < \infty$.
\end{theorem}

\begin{proof}
The proof is very similar to that of Theorem
\ref{product-prior-is-convex-ellp}. The only difference is
that the sum $\sum_{k=1}^\infty | \gamma_k \xi_k| < \infty$ a.s. following
Lemma \ref{sum-of-positive-random-variables}.
% } It follows from triangle inequality and the assumption
%   that $\{ x \}$ is a normalized basis that $\| u \|_X \le
%   \sum_{i=1}^\infty |\gamma_i\xi_i|$. Now the sum on the right hand
%   side is bounded a.s. following the assumption of the
%   theorem and Lemma \ref{sum-of-positive-random-variables} and
%   therefore $\| u\|_X <\infty$ a.s. by the dominated
%   convergence theorem. 
% {\it (ii) } Let $\tilde{\nu} = \bigotimes_{i =1}^\infty \mu_i$ denote the quasi-measure that is 
% generated by the countable product of the $\mu_i$ on $\reals^\infty$ (the countable product of $\reals$).
% By \cite[Thm.~7.6.2]{bogachev2} $\tilde{\nu}$ has a unique extension to a Radon measure $\nu$
% on $\reals^\infty$. Now consider the operators
% $$
% Q_N( \cdot ; \{ \gamma \}) : \reals^\infty \to X \qquad Q_N( \{ c \}; \{ \gamma \} ) = \sum_{i=1}^N c_i \gamma_i x_i. 
% $$
% It is easily verified that $Q_N$ are linear and continuous for each $N
% > 0$ and so taking  $\mu_N = \nu \circ Q_N^{-1}$, 
% $\mu_N$ will be convex measures on $X$ by Corollary
% \ref{convexity-preservation}.
\end{proof}

In general, the choice of
the conditions on $\{ \gamma_k \}$ and $\{ \xi_k \}$ is an exercise in convergence of series of
independent random variables. We refer the reader to the set of
lecture notes \cite[Section~2]{stuart-bayesian-lecture-notes} for more
examples of this approach. 

We can further refine our result on the support of the measure $\mu$ that is
induced by $u$. Assumption \ref{assumption-on-X} implies that the space $X$
is isomorphic to the sequence space (see {\cite[Thm.~4.12]{heil-basis}})
\begin{equation}
   \label{sequence-space}
   \begin{aligned}
    & W := \left\{ \{w_k \} \in \reals : \sum_{k=1}^\infty w_k x_k  \text{ converges in } X \right\}, \\
  &\| \{w\} \|_W = \sup_{n \in \integers} \left\| \sum_{k=1}^n w_k x_k \right\|_X,
   \end{aligned}
 \end{equation}
and the synthesis operator $T: \{w_k\} \mapsto
\sum_{k=1}^\infty w_k x_k$ is an isomorphism between $X$ and $W$. To
this end, we can show that the support of
$\mu$ is a Hilbert space under mild
conditions.

\begin{theorem}\label{hilbert-space-support}
Let $u$ be defined as in \eqref{product-prior-sample} and let $\mu$ be
the probability measure that is induced by this random variable on $X$. Now
suppose that
the $\{ \xi_k\}$ are i.i.d.\ convex random variables on $\reals$. Also, let $\{ |\gamma_k|^2 \} \in
\ell^1(\reals)$.
 Then $\mu$ is
  concentrated on a separable Hilbert space $H \subseteq X$ where
$$
H := \left\{ u \in X :  \sum_{k=1}^\infty c_k^2 \hat{u}_k ^2 <\infty\right\}
\quad \text{with inner product} \quad
\langle u, v \rangle_H :=\sum_{k=1}^\infty c_k^2 \hat{h}_k \hat{v}_k.
$$
Here, $\{ \hat{u}_k \} := T^{-1}u$ denotes the sequence
of basis coefficients of an element $u \in X$ and
 $\{
c_k \}$ is a fixed sequence that decays sufficiently fast
so that $\sum_{i=1}^\infty c^2_k \gamma^2_k \xi_k^2 < \infty$ a.s. In
particular, it is sufficient if $\{ c_k \} \in \ell^2(\reals)$.
\end{theorem}
\begin{proof}
We only consider the case when $\{ c_k \} \in \ell^2(\reals)$. First,
by Corollary \ref{product-prior-is-convex-iid} we know that $\| u\|_X
<\infty$ a.s. and so we can work directly with the random sequence $\{
\hat{u}_k\} $ and the measure that
this sequence will induce on $W$. This measure will be equivalent to
$\mu \circ T$ and its support is isomorphic to the support of
$\mu$. 
Define the centered random variables $\zeta_k = |\xi_k|^2  -  \EE
|\xi_k|^2$ and write 
$$
\sum_{k=1}^\infty c_k^2 \gamma_k^2 \xi_k^2 = \sum_{k=1}^\infty c_k^2
\gamma_k^2 \zeta_k + \sum_{k=1}^\infty c_k^2 \gamma_k^2 \EE |\xi_k|^2.
$$
Since the
$\mu_k$ are convex they have bounded moments of all orders and so
we can
define the constant
$C = \VV \zeta_1 $.
 Using H\"{o}lder's inequality and the fact that $\ell^p(\reals)
 \subset \ell^q(\reals)$ for $1\le p < q < \infty$ 
we have
$$
\begin{aligned}
\sum_{k=1}^\infty \VV |c_k^2 \gamma_k^2| \zeta_k&=  C
\sum_{k=1}^\infty |c_k|^4 |\gamma_k|^4  \le C \| \{
|c_k|^4 |\gamma_k|^4 \}  \|_{\ell^1} 
\\ & \le C\| \{ |c_k|^4\} \|_{\ell^2} \| \{|\gamma_k|^4\}
\|_{\ell^2} = C\|  \{c_k\} \|_{\ell^8}^4 \| \{|\gamma_k|^2\}
\|_{\ell^4}^2 \\
& \le C\|  \{c_k\} \|_{\ell^8}^4 \| \{|\gamma_k|^2\}
\|_{\ell^1}^2 \\
& \le C\|  \{c_k\} \|_{\ell^2}^4 \| \{|\gamma_k|^2\}
\|_{\ell^1}^2 < \infty.
\end{aligned}
$$
By 
Kolmogorov's two series theorem
\cite[Lemma~3.16]{kallenberg} we have that $\sum_{k=1}^\infty c_k^2
\gamma_k^2 \zeta_k < \infty $ a.s.
Using a similar argument as above we can show that $\sum_{k=1}^\infty c_k^2
\gamma_k^2 \EE | \xi_k|^2 < \infty$ and so
$\sum_{k=1}^\infty c_k^2 \gamma_k^2 \xi_k^2 <\infty$ a.s. Observe
that our method of proof only requires $\{ c_k\} \in \ell^8(\reals) $
and the assumptions on $\{c_k\}$ can be relaxed.

% Now definte the random variables ${w}_N = \sum_{i=1}^Nc_i^2 \gamma_i^2 \xi_i^2$ and 
% ${w}_M = \sum_{i=1}^Mc_i^2 \gamma_i^2 \xi_i^2$ for $0 < N < M <
% \infty$. Then 
% $$
% |w_M - w_N | = \sum_{i=N}^M c_i^2 \gamma_i^2 \xi_i^2 
% \le \sum_{i=N}^\infty c_i^2 \gamma_i^2 \xi_i^2.
% $$
% It follows from the dominated convergence theorem that the right hand
% side will vanish a.s. in the limit as $M,N \to \infty$ and so $\{
% w_N\}_{N=1}^\infty$ is Cauchy a.s.
%  defined as in \eqref{product-prior-sample}, 
% then
%   following Theorem \ref{product-prior-is-convex}(i) we can define 
% the random variable
%   $ \{ w \} :=  T^{-1} u = \{ \gamma_i \xi_i
%   \}_{i=1}^\infty \in W$ which is distributed according to $\mu \circ
%   T$. Recall that since $\xi_i$ are convex random variables then $\EE
%   \exp(\epsilon_i |\xi_i|) < \infty$ for sufficiently small constants
%   $\epsilon_i>0$. This implies that all the moments of each $\xi_i$
%   are bounded. It then follows from Kolmogorov's two series theorem
%   that one can always find a sequence $\{ c \}$ that satisfies the
%   above condition.
\end{proof}

At the end of this section we note that the construction of the prior via
\eqref{product-prior-sample} can be generalized to the setting where
one starts from a Banach space $X$ and constructs a
measure that is supported on a larger space that contains $X$. For
example, we can start with $X = L^2(\mbb{T})$ and construct a measure
on a Sobolev space with a negative index.

 To this end, consider a random variable of the form
\begin{equation}
  \label{product-prior-sample-rough}
  u = \sum_{k=1}^\infty \xi_k x_k
\end{equation}
where $\{\xi_k \}$ are i.i.d.\  and centered convex random variables and
$\{ x_k \}$ is a normal basis in $X$. 
It is easy to see that
the samples no longer belong to $X$ almost surely. However, they
belong to a larger space $\tilde{X}$. Let $\tilde{W}$ be the space of
real valued sequences $\{ w_k \}$ so that
$$
\| \{ w  \} \|_{\tilde{W}} := 
\left\|\sum_{k=1}^\infty c_k w_kx_k \right\|_X < \infty.
 $$
Here we take $\{ c_{k} \} $ to be a fixed sequence that decays
sufficiently fast for $\sum_{k=1}^\infty c_k \xi_k$ to converge a.s. A
convenient choice would be $\{ c_k^2 \} \in \ell^1(\reals)$. Now define $\tilde{X}$ to be the image of $\tilde{W}$
under the synthesis map 
$$
T: \tilde{W} \to \tilde{X} \qquad T(\{ w \}) = \sum_{i=1}^\infty w_i x_i.
$$
It is straightforward to check
that the
measure $\mu$ that is induced by \eqref{product-prior-sample-rough} is
a convex measure on $\tilde{X}$ using similar arguments to the proof
of Theorem \ref{product-prior-is-convex-ellp}. It is also clear that
$X \subset \tilde{X}$.

\section{Practical considerations and examples}\label{sec:examples}
In this section we will
discuss certain problems that are particularly
interesting in applications. Throughout this section we consider the
additive noise model of Section \ref{sec:introduction}:
$$
y = \mcl{G}(u) + \eta, \qquad y \in Y, \quad u \in X, \quad \mcl{G} :
X \to Y
$$
where $Y = \reals^m$ for some integer $m$ and $\eta \sim \mcl{N}(0,
\pmb{\Gamma})$ where $\pmb{\Gamma}$ is a positive-definite matrix and
$X$ is a Banach space that satisfies Assumption \ref{assumption-on-X}.
We have already
seen that under these assumptions the likelihood potential has the
form 
\begin{equation}
  \label{gaussian-likelihood-2}
  \Phi(u;y) = \frac{1}{2} \left\| \pmb{\Gamma}^{-1/2} ( \mcl{G}(u) - y
    ) \right\|_2^2 =: \frac{1}{2} \left\|\mcl{G}(u) - y
     \right\|_{\pmb{\Gamma}}^2 .
\end{equation}
We can now reduce Assumption \ref{assumption-on-likelihood}
to a smaller set of assumptions on $\mcl{G}$ (resp. $\Phi$).
\myhl{Note that, with minor modifications, the  proof of the following theorem
is applicable to other types of additve noise models such as Laplace noise. 
}

\begin{theorem}\label{forward-model-assumptions}
Suppose that $X$ is a Banach space and assume that the operator
$\mcl{G}: X \to \reals^m$ satisfies the following conditions: 
\begin{enumerate}[(i)]
\item For an $\epsilon > 0$ there exists a constant $M = M(\epsilon) \in
  \reals$ such that $\forall u \in X$ 
$$
\| \mcl{G}(u)  \|_{\pmb{\Gamma}} \le \exp( \epsilon \| u\|_X + M ).
$$
\item For every $r > 0$ there is a constant $K = K(r) > 0$ such that
  for all $u_1, u_2 \in X$ with $\max\{ \| u_1 \|_X, \| u_2 \|_X \} <
  r$
$$
\| \mcl{G}(u_1) - \mcl{G}(u_2) \|_{\pmb{\Gamma}} \le K \| u_1 - u_2 \|_X.
$$
\end{enumerate}
Then the likelihood potential $\Phi$ given by
\eqref{gaussian-likelihood-2} satisfies the conditions of Assumption \ref{assumption-on-likelihood}.
\end{theorem}

\begin{proof}
  We will go through the conditions of Assumption \ref{assumption-on-likelihood} one by one. First,
  it is clear that by taking $\alpha_1 = 0$ and $M(r) = 0$
  Assumption \ref{assumption-on-likelihood}(i) is satisfied. Now fix an $r>0$ and consider $u \in
  X$ and $y \in Y$ so that $\max\{ \| u \|_X, \| y \|_Y \} < r$. It
  then follows from condition (i) on $\mcl{G}$ that for 
  $\epsilon > 0$ 
$$
\Phi(u;y) \lesssim  \| y \|_{\pmb{\Gamma}}^2 + \| \mcl{G}(u)
\|_{\pmb{\Gamma}}^2 + \| y \|_{\pmb{\Gamma}} \| \mcl{G}(u)
\|_{\pmb{\Gamma}}  \lesssim r^2 + \exp( 2 \epsilon r + M(\epsilon))  + r
\exp(\epsilon r + M(\epsilon)) = K(r)
$$
which gives Assumption \ref{assumption-on-likelihood}(ii). Once again fix $r > 0$ and consider
$u_1, u_2 \in X$ and $y \in Y$ so that $\max\{ \|u_1 \|_X, \| u_2
\|_X, \|y \|_Y \} < r$. Then, using conditions (i) and (ii) on $\mcl{G}$, it is easy to verify that
$$
\begin{aligned}
| \Phi(u_1, y) - \Phi(u_2, y) | & \le \left|  \| \mcl{G}(u_1)
\|_{\pmb{\Gamma}}^2  - \| \mcl{G}(u_2) \|_{\pmb{\Gamma}}^2 + 2 \| y
\|_{\pmb{\Gamma}}^2 \| \mcl{G}(u_1) - \mcl{G}(u_2) \|_{\pmb{\Gamma}}
\right| \\ &\le K(r) \| \mcl{G}(u_1) - \mcl{G}(u_2) \|_{\pmb{\Gamma}}.
\end{aligned}
$$
Finally, Assumption \ref{assumption-on-likelihood}(iv) follows from a similar argument using both
conditions on $\mcl{G}$.
\end{proof}\\
\begin{remark}
~If condition (i) of Theorem \ref{forward-model-assumptions}\
is satisfied for any $\epsilon > 0$ then $\Phi$ satisfies Assumption
\ref{assumption-on-likelihood}(iv) for any constant $\alpha_2 > 0$.
In particular this is true if $\| \mcl{G} \|_{\pmb{\Gamma}}$ is polynomially bounded in $\| u
  \|_X$.
\end{remark}
\begin{remark}
  ~If $\mcl{G}$ is a bounded linear operator from $X$ to $\reals^m$ then it satisfies the 
 conditions of Theorem \ref{forward-model-assumptions} for any constant $\epsilon > 0$.
\label{bounded-linear-operator-assumption}
\end{remark}

\subsection{Consistent discretization by projection}
In Section \ref{sec:well-posedness} we mentioned that  an important aspect of solving a
Bayesian inverse problem is to have a consistent approximation. In practice we often use a numerical algorithm
in order to extract interesting statistics from
the posterior (be it sampling or optimization).
Since we cannot compute the solution of the infinite dimensional problem we need
to discretize it. In the case of the product priors of Section
\ref{construction-of-convex-priors} the most convenient method of discretization is to truncate
the spectral expansion of the samples.

Suppose that $X$
is a Banach space that satisfies Assumption
\ref{assumption-on-X}, i.e. it has a unconditional normalized basis $\{ x_k
\}$. Furthermore, let $( X_N, \| \cdot \|_{X} )$
for $N =1,2, \cdots$ be a
sequence of finite dimensional linear subspaces of $X$ (not
necessarily nested) where each $X_N$ is simply the span of a finite
number of the $x_k$. 
Then for each $N$ the space can be factored as $X = X_N \bigoplus
X_N^{\bot}$ i.e. every element $u \in X$ can be written as 
$u = u_N + u_N^\bot$ where $u_N \in X_N$ and $u_N^\bot \in X_N^\bot :=
(X \setminus X_N)$. Now suppose that there are projection
operators $P_N: X \to X_N$ so that 
$$
P_N u = u_N \qquad \text{and} \qquad \| u - P_Nu \|_X \le \psi(N) \to 0 \qquad as\qquad  N \to \infty \qquad \forall u \in X.
 $$
A simple example of this setting is to let $X_N = \text{span} \{ x_k
\}_{k=1}^N$ and take $P_Nu = \sum_{k=1}^N c_k(u) x_k$ where
$\{c_k(u)\}$ are the basis coefficients of $u \in X$. \myhl{If $X$ is a Hilbert space with an orthonormal basis then one can
take $P_N$ to be the usual Galerkin projection.} 

Now consider the Bayesian
inverse problem given by \eqref{bayes-rule} with a
likelihood potential $\Phi$ that satisfies the conditions of
Assumption \ref{assumption-on-likelihood}.
 Suppose that
$\mu_0$ is a convex prior measure that satisfies
Assumption \ref{assumption-on-prior} with a sufficiently large
constant 
$\kappa$. In addition, suppose that for each $X_N$
the prior can be factored as
$$
\mu_0 = \mu_N \otimes \mu_N^{\bot} 
$$
where $\mu_N$ is a measure on $X_N$ and $\mu_N^{\bot}$ is a
measure on $X_N^{\bot}$. An example of such a prior $\mu_0$ is our
construction of the product convex measures in Section
\ref{construction-of-convex-priors}. Now consider the collection of
approximate posterior
measures $\mu_N^y$ given by 
$$
\frac{\dd \mu_N^y(u)}{ \dd \mu_0(u)}  = \frac{1}{Z_N(y)} \exp( - \Phi(
P_Nu)) \qquad \text{where} \qquad Z_N(y) = \int_X \exp( - \Phi(
P_Nu)) \dd \mu_0(u).
$$
% Furthermore, define the measure 
% \begin{equation*}
%   \label{discretized-posterior-definition}
% \frac{\dd \nu_j^y}{\dd
%   \mu_j} = \frac{1}{Z_j(y)} \exp\left( - \Phi( w; y) \right)
% \qquad Z_j(y) = \int_{X_j} \exp(- \Phi( w; y) )  \dd\mu_j(w)
% \end{equation*}
Observe that for every set $A \in \mcl{B}(X)$, using Fubini's theorem \cite[Theorem~3.4.1]{bogachev2}
we have
$$
\begin{aligned}
\mu^y_N(A) &= \int_{A} \frac{1}{Z_N(y)} \exp( - \Phi( P_N u;y) \dd
\mu_N^\bot \otimes \dd\mu_N(u)  \\  & = 
 \int_{X_N}\int_{A_{v}} \frac{1}{Z_N(y)} \exp( -\Phi( v;y ) )\dd
 \mu_N^\bot(w)  \dd \mu_N(v) \\ & = 
 \int_{X_N}  \mu_N^\bot(A_v) \left(\frac{1}{Z_N} \exp( -\Phi( v;y ) )\right)
 \dd \mu_N(v)
\end{aligned}
$$
Where $A_v = \{ w \in X_N^\bot : v + w \in A \}$. 
This implies that the posterior measure $\mu_N^y$ can be factored as 
$\mu^y_N = \tilde{\mu}^y_N \otimes \mu_N^\bot$ where 
$$
\frac{\dd \tilde{\mu}_N^y}{\dd
  \mu_N} = \frac{1}{Z_N(y)} \exp\left( - \Phi( v; y) \right).
$$
Thus, the posterior measure inherits the product structure of the
prior and it only differs from the prior on the subspace $X_N$.

It is straightforward to check that the measures $\mu_N^y $ are
well-defined and absolutely continuous with respect to the prior
following our assumptions on $\Phi$ and $\mu_0$. However, it remains
for us to show that the measures will converge to the posterior
$\mu^y$ in an appropriate sense.

\begin{lemma}\label{convergence-of-discrete-likelihood}
 Consider the above setting where the posterior and the prior have the
 prescribed 
 product structures and the $X_N$ are linear subpaces of $X$. Suppose
 that $\Phi$ is given by \eqref{gaussian-likelihood}
 and $\Phi_N(u;y) := \Phi(P_N u; y)$. Furthermore, suppose that for
 any $\epsilon > 0$ there exist constants $K(\epsilon) > 0$ and
 $M(\epsilon) \in \reals$ so that 
 \begin{enumerate}[(i)]
 \item $ \forall u \in X$, $ \| \mcl{G}(u) \|_2 \le \exp( \epsilon
   \| u\|_X + M).$
\item $
\| \mcl{G}(u_1) - \mcl{G}(u_2) \|_{2} \le  K \exp( \epsilon
\max\{ \| u_1 \|_X, \| u_2 \|_X \} ) \| u_1 - u_2 \|_X.
$
 \end{enumerate}
 % Now assume that $r = \max\{ \|
 % u\|_X, \sup_j \| P_j u \|_X\}$ is bounded for all $u \in X$
 % and $y \in Y$.
Then there exists a constant $C(\epsilon) > 0$ such that
$$
| \Phi(u;y) - \Phi_N(u;y) | \le C \exp( \epsilon \| u \|_X ) \| u -
P_N u \|_X.
 $$
%  Then there exists a constant $D$ that is independent of $N$ and
% $$
% d_{\text{Hell}} ( \mu^y, \mu_N^y ) \le D \| u - P_N u \|_X.
% $$
\end{lemma}
\begin{proof}
Let
  $\mcl{G}_N(u) := \mcl{G}(P_Nu)$ and
  fix an $\epsilon > 0$. Then using the elementary identity $(a^2 -
  b^2) = (a+b)(a-b)$, the triangle inequality and the fact that the
  data $y$ is a fixed vector, we have:
$$
\begin{aligned}
| \Phi(u;y) - \Phi_N(u;y) | &\le \frac{1}{2} \| 2y - \mcl{G}(u) -
\mcl{G}_N(u) \|_{\pmb{\Gamma}}\: \| \mcl{G}(u) - \mcl{G}_N(u)
\|_{\pmb{\Gamma}} \\  
& \le
\frac{1}{2} (\| 2y \|_{\pmb{\Gamma}}  + \|   \mcl{G}(u)
\|_{\pmb{\Gamma}}  + \|
\mcl{G}(P_N u) \|_{\pmb{\Gamma}} )\: \| \mcl{G}(u) - \mcl{G}_N(u)
\|_{\pmb{\Gamma}} \\
& \le \frac{1}{2} (\| 2y \|_{\pmb{\Gamma}}  + C \exp(\epsilon \| u
\|_X + M(\epsilon) )\: \| \mcl{G}(u) - \mcl{G}_N(u)
\|_{\pmb{\Gamma}} \\
& \le C(\epsilon) \exp\left( \epsilon \| u \|_X \right) \| u - P_N u \|_X
\end{aligned}
$$
\end{proof}
\begin{corollary}\label{convergence-of-posterior-with-projection}
Consider the setting of Lemma \ref{convergence-of-discrete-likelihood}
 above and let $\psi(N)$ be a
function such that $\psi(N) \to 0$ as $N \to \infty$.
If  for every $\epsilon
  \ge 0$ there exists a constant $M(\epsilon) \in \reals$ so that 
$$
\| u - P_N u \|_X \le \exp( \epsilon \| u \|_X  + M ) \psi(N) 
$$
then there exists a constant $D$ independent of $N$ so that
$$d_{H}(\mu^y , \mu^y_N) \le D \psi(N).$$
In other words, the error in the Hellinger metric decays at the same rate as the approximation error $\| u - P_Nu\|_X$
as $N \to \infty$.
\end{corollary}
\begin{proof}
  Note that our Assumption on continuity of $\mcl{G}$ in Lemma \ref{convergence-of-discrete-likelihood} is
stronger than the conditions of Theorem
\ref{forward-model-assumptions}. Then it follows from Theorem \ref{existence-uniqueness} that the posterior
$\mu^y$ and the approximations $\mu^y_N$ are well-defined. Finally, the
desired result follows directly from Lemma
\ref{convergence-of-discrete-likelihood} and Theorem \ref{stability}.

\end{proof} \\ 
\subsection{Example 2: Deconvolution with a Laplace prior}
We now return to Example 2 where we considered the deconvolution of an
image on the circle. Recall that $X =
L^2(\mathbb{T})$ and the data $y$ was generated by 
considering a fixed function $v \in L^2(\mathbb{T})$ and letting $y =
\mcl{G}(v) + \eta$ where $\eta \sim \mcl{N}(0, \sigma^2 \mb{I})$ and  the forward map is given by 
$$
\mcl{G}: L^2(\mbb{T}) \to \reals^m, \qquad \mcl{G}(u) :=  S( g \ast u).
$$
Here $g$ is a $C^\infty(\mbb{T})$ kernel and $S$ is a bounded linear
operator that collects point values of the convolved image.
Since the measurement noise is additive and Gaussian then
the likelihood potential $\Phi$ is quadratic
\eqref{deconvolution-likelihood}.
The prior measure in \eqref{deconvolution-prior} has the form
$$
u = \sum_{k \in \mbb{Z}} \gamma_k \xi_k \psi_k
$$
where $\{\xi_k\}$ are i.i.d.\  $\text{Lap}(0,1)$ random variables and $\gamma_k
= (1 + |k|^2)^{-5/4}$. Thus, $\{ \gamma_k^2 \} \in
\ell^1(\reals)$. Furthermore, we took
$\psi_k(x) = (2\pi)^{-1/2}e^{-2\pi i k x}$ which form an orthogonal basis in $L^2(\mbb{T})$ and
so by Corollary
\ref{product-prior-is-convex-iid} the random variable $u \in
L^2(\mbb{T})$ a.s. and $\mu_0$ is a convex measure.

Next, it follows from Young's inequality \cite[Thm.~13.8]{heil-basis} that $(g \ast \cdot):L^2(\mathbb{T}) \to L^2(\mathbb{T})$ is a bounded linear 
operator and furthermore, $(g \ast u) \in
C^\infty(\mathbb{T})$ for all $u \in
L^2(\mathbb{T})$.
Since pointwise evaluation is a bounded linear functional on
$C^\infty(\mathbb{T})$ then the forward map $\mcl{G}:
L^2(\mathbb{T}) \to \reals^m$ is a bounded linear operator. Therefore,
by Remark \ref{bounded-linear-operator-assumption} and Theorem
\ref{forward-model-assumptions} we know that the likelihood potential
$\Phi$
satisfies
Assumption \ref{assumption-on-likelihood}. Putting this together with
Theorem \ref{convex-bayesian-inverseproblems} implies that this
deconvolution problem is well-posed.

Now we will consider a consistent approximation of this problem.
Given $w \in L^2(\mathbb{T})$, we let $\{ \hat{w}_k\}_{k \in \mathbb{Z}}$ denote 
the Fourier modes of $w$ i.e.
$$
\hat{w}_k = (2 \pi)^{-1/2} \int_0^1 w(x) \:e^{-2\pi i k x} dx, \qquad \forall w
\in L^2(\mathbb{T}) \quad \text{and} \quad k \in \mathbb{Z}.
$$
We also define the Sobolev space $H^2(\mathbb{T})$
as 
$$
H^2(\mathbb{T}) := \left\{w  \in L^2(\mathbb{T}), \qquad \|w\|_{H^2}:=\sum_{k\in \mathbb{Z}}
( 1+ | k|^2)^2 | \hat{w}_k|^2 < \infty \right\}.
$$
Since the prior samples $u \in L^2(\mbb{T})$ a.s.
then we can consider their Fourier modes $\hat{u}_k = \gamma_k
\xi_k$. We can show that in fact such a sample $u\in H^2(\mathbb{T})$ a.s., and are therefore in $C^1(\mathbb{T})$. To this end, note that
by substituting the values of the $\gamma_k$ we can write
$$\|u\|_{H^2}:=
\sum_{k \in \mbb{Z}} ( 1 + | k|^2 )^2  \frac{1}{(1 + |k|^2)^{5/2}}
|\xi_k|^2 = \sum_{k \in \mbb{Z}} ( 1 + | k|^2 )^{-1/2} 
|\xi_k|^2.
$$
This sum will converge almost surely following Kolmogorov's two series
theorem \cite[Lemma~3.16]{kallenberg} and dominated convergence and so 
$u \in H^2(\mathbb{T})$
a.s.  This means that the samples can be
approximated by truncation of their Fourier expansions.

Suppose that we discretize the prior by truncating the series in
\eqref{deconvolution-prior} and define the projections
$$
P_N u = \sum_{k=-N}^{N-1} \gamma_k \xi_k \psi_k.
$$
Since $u \in H^2(\mbb{T})$ a.s. then
\begin{equation}\label{deconvolution-discretization-convergence-rate}
\| u - P_N u \|_{L^2(\mbb{T})} \lesssim \| u\|_{L^2(\mbb{T})} \frac{1}{N^2}.
\end{equation}
Now let $\mu^y_N$ denote an
approximation to the posterior that is obtained by defining
$\mcl{G}_N(u) := \mcl{G}(P_Nu)$. Since $\mcl{G}$ is a bounded linear
operator then
it satisfies the conditions of
Lemma~\ref{convergence-of-discrete-likelihood}. Furthermore, the
likelihood is quadratic and so by
Corollary~\ref{convergence-of-posterior-with-projection} 
and \eqref{deconvolution-discretization-convergence-rate} we have that 
$$
d_{H}(\mu^y, \mu^y_N) \lesssim \frac{1}{N^2}.
$$

\subsection{Example 3: Deconvolution with a hierarchical prior} \myhl{We now
return to Example 3 that was first introduced in Section \ref{sec:introduction}. 
Recall that the problem is in an identical setting as in Example 2
above with the exception that the prior measure is now induced by the
random variable }
\begin{equation}
\label{deconvolution-hierarchical-prior}
u = \sum_{k \in \mbb{Z}} \gamma_k \zeta_k \xi_k \psi_k
\end{equation}
where $\gamma_k = (1+ |k|^2)^{-1}$ and the $\psi_k$ are the
Fourier basis functions on $L^2(\mbb{T})$. Here, $\{ \zeta_k \}$ and
$\{ \xi_k \}$ are two sequences of i.i.d.\  random variables so that
$\zeta_1 \sim \text{Gamma}(2,1)$ and $\xi_1 \sim \mcl{N}(0,1)$.
The main difficulty
in the study of this problem is the fact that even though $\xi_k$ and
$\zeta_k$ have convex distributions their product may not be
convex. To get around this issue and prove the well-posedness of the
problem we shall cast it directly on the sequence space and
work with a nonlinear forward map.

First, note that $\EE \zeta_k\xi_k = \EE \zeta_k \EE \xi_k =0$ and 
$\VV \zeta_k\xi_k = (\VV \zeta_k)(\VV \xi_k)  +
(\VV\zeta_k) \EE \xi_k + (\VV  \xi_k) \EE \zeta_k < \infty$. We can use the same technique as in the proof of Theorem
\ref{product-prior-is-convex-ellp}(i) to show that $u \in
L^2(\mbb{T})$ a.s. Furthermore, the random sequences $\{
\sqrt{\gamma}_k \zeta_k\}$ and $\{ \sqrt{\gamma}_k \xi_k \}$ belong to
the sequence space $\ell^1$ a.s. and so they belong to $\ell^2$ as
well. We now consider the product space $\ell^2 \otimes \ell^2 =\{ 
\{ c_k, b_k \} : \{ c_k\} \in \ell^2, \{ b_k \} \in \ell^2 \}$
equipped with the norm $\| \{ c_k, b_k \} \|_{\ell^2 \otimes \ell^2} := \max \{ \| \{c_k\}
\|_{\ell^2}, \| \{ b_k \} \|_{\ell^2} \}  $ along with the probability
measure $\mu_0$ that is induced by $\{
\sqrt{\gamma}_k \zeta_k, \sqrt{\gamma}_k \xi_k \}$ on the product
space. We will take $\mu_0$ to be our prior.
 Note that $\mu_0$ can be obtained from the product of two convex
measures on $\ell^2$ and so it is a convex measure by Corollary
\ref{convexity-preservation}(ii).

Define the operator 
$$
Q : \ell^2 \otimes \ell^2 \to L^2(\mbb{T}) \qquad 
Q( \{ c_k, b_k \} ) = \sum_{k\in \mbb{Z}} c_k b_k \psi_k. 
$$
Using the triangle and H\"{o}lder's inequalities we can easily check that
\begin{equation*}
%\label{Q-is-bounded}
\|Q( \{ c_k, b_k \} ) \|_{L^2(\mbb{T})} \le \|  \{ c_k, b_k \}
\|_{\ell^2 \otimes \ell^2}^2
\end{equation*}
This bound together
with the fact that $\mcl{G}: L^2(\mbb{T}) \to \reals^m$ is a bounded
linear operator implies that 
\begin{equation*}
\| \mcl{G} \circ Q(\{ c_k, b_k \})  \|_2 \le C \| \{ c_k, b_k \}
\|_{\ell^2 \otimes \ell^2}^2.
\end{equation*}
Since $\|\mcl{G}\circ Q\|_2$ is bounded by a quadratic function of $\|
\{ c_k, b_k\} \|_{\ell^2 \otimes \ell^2}$ 
 then for every $\epsilon >0$ we
can always find $M(\epsilon) >0$ so that 
\begin{equation}\label{Q-is-exponentially-bounded}
\| \mcl{G} \circ Q(\{ c_k, b_k \} )  \|_2 \le C \exp( \epsilon \| \{ c_k,
b_k \} \|_{\ell^2 \otimes \ell^2} + M(\epsilon)). 
\end{equation}
Now consider $\{ c_k, b_k\}$ and $\{ g_k, f_k \}$ in $\ell^2 \otimes \ell^2$. Using Parseval's identity we can
write 
$$
\begin{aligned}
\| Q( \{ c_k, b_k\}) - Q(\{ g_k, f_k\} ) \|_{L^2(\mbb{T})}  &=  \left( \sum_{k
  \in \mbb{Z}} | c_kb_k - g_kf_k |^2 \right)^{1/2} \\
&= \left(\sum_{k\in \mbb{Z}} | c_k (b_k - f_k) + f_k (c_k -
g_k)|^2 \right)^{1/2}\\
& \le C \max( \| \{c_k, b_k \} \|_{\ell^2 \otimes \ell^2}, \| \{g_k, f_k \} \|_{\ell^2 \otimes \ell^2})
\| \{ c_k - g_k, b_k-f_k \} \|_{\ell^2 \otimes \ell^2}.
\end{aligned}
$$ 
Putting this result together with the fact that $\mcl{G} :
L^2(\mbb{T}) \to \reals^m$ is a bounded linear operator implies that
 for any $r > 0$ such that $\max( \| \{c_k, b_k \} \|_{\ell^2 \otimes \ell^2},
\| \{g_k, f_k \} \|_{\ell^2 \otimes \ell^2}) < r$ there is a constant $K(r)> 0$ so that
\begin{equation}
  \label{Q-is-Lipschitz}
  \| \mcl{G} \circ Q( \{ c_k, b_k\}) - \mcl{G} \circ Q(\{g_k, f_k\}) \|_2 \le
  K(r) \| \{ c_k - g_k , b_k- f_k \} \|_{\ell^2 \otimes \ell^2}.
\end{equation}
This bound along with \eqref{Q-is-exponentially-bounded} and
Theorem \ref{forward-model-assumptions} implies that the underlying
likelihood potential \myhl{$\Phi = \frac{1}{2 \sigma^2} \| \mcl{G}(u) - y\|_2^2$} satisfies the conditions of Assumption
\ref{assumption-on-likelihood} for any choice of $\alpha_1 , \alpha_2
> 0$. Since the prior measure
$\mu_0$ is convex
 this inverse
problem is well posed by Corollary \ref{convex-bayesian-inverseproblems-poly-bounded}.
\subsection{Example 5: Source inversion in atmospheric dispersion}
In this example we consider the problem of estimating the source term
in a parabolic PDE from linear measurements of the solution. This
problem is closely related to the inverse problem of estimating the
sources of emissions in an atmospheric dispersion model
\cite{hosseini-lead-inverse} and we shall present this example in that
context. Let $D\subset \reals^3$ be a smooth and
connected domain and define $\Omega
:= D \times (0, T]$ for some constant $T > 0$. Now consider the PDE 
% \begin{equation}
%   \label{parabolic-PDE}
% \left\{
%   \begin{aligned}
%    & \partial_t c - \nabla \cdot {w}(x,t) c - \nabla \cdot ( \mb{K}(x,t) \nabla c)  = u,
%   \qquad &&\text{in} \qquad D \times (0,T)  \\
%   &c(x,t) = 0  \qquad &&\text{on} \qquad \partial D \times(0,T),\\ 
% & c(x,0) = 0.
%   \end{aligned} \right.
% \end{equation}
% Here, ${w}$ is a vector with real valued entries $w_i(x,t)$ for $ i \in \{ 1,
% 2, 3\} $ and
% $\mb{K}$ is a square matrix of size 3 with real valued entries
% $k_{ij}(x)$ for $i,j \in \{ 1, 2, 3\}$. In the context of atmospheric
\begin{equation}
  \label{parabolic-PDE-expanded}
\left\{
  \begin{aligned}
   & \partial_t c - g(x,t) c - \sum_{i=1}^3 a_i(x,t) \partial_i c -
   \sum_{i,j=1}^3 b_{ij} \partial_{ij} c  = u,
  \qquad &&\text{in} \qquad D \times (0,T)  \\
  &c(x,t) = 0  \qquad &&\text{on} \qquad \partial D \times(0,T),\\ 
& c(x,0) = 0.
  \end{aligned} \right.
\end{equation}
where $\partial_ic$ is the shorthand notation for $\frac{\partial
  c}{\partial x_i}$ and $\partial_{ij} c = \partial_i \partial_j c$.
In the context of atmospheric
dispersion modelling $c(x,t)$ is the pollutant concentration, $u(x,t)$ is
the source term and $g, a_i, b_{ij}$ coefficients are used to model the wind
velocity field and the eddy diffusivity coefficients \cite{hosseini-dispersion,
  seinfeld}.
% Now consider the function space 
% $$
% V := \{ v \in L^2(\Omega) : \nabla v \in L^2(\Omega)  \quad \text{and}
% \quad v( \cdot, t) \in L^2(D) \: \forall t \in (0,T)\} 
% $$
% and endow this space with the norm 
% $$
% \| v \|_V := \| \nabla v \|_{L^2(\Omega)} + \sup_{t \in (0,T)} \|
% v(\cdot, t) \|_{L^2(D)}.
% $$
To this end, we have the following result concerning the existence
and uniqueness of the solution to \eqref{parabolic-PDE-expanded} (see
\cite[Section~7.1]{Evans} or \cite[Theorem~11.3 and Example~11.5]{Renardy} for a proof).
\begin{theorem}\label{parabolic-existence-uniqueness}
  Suppose that $\Omega = D \times (0,T)$ where $D \subset \reals^3$ is
  defined as above and $T > 0$. Also, assume that $u \in L^2(\Omega)$ and $g, a_i,
  b_{ij}$ are in $C(\Omega)$ for $i,j \in
  \{1,2,3\}$. Furthermore, assume that $b_{ij} = b_{ji}$ and there
  exists a uniform constant $K > 0$ such that 
$$\sum_{i,j = 1}^3 b_{ij}(x,t) y_i y_j > K \sum_{i=1}^3 y_i^2 \qquad \forall y_i \in \reals \setminus \{0\} 
\quad \text{and} \quad \forall (x,t) \in \Omega.$$
%  and there
%   are positive constants $k_1, k_2$ and $k_3$ such that
%   \begin{enumerate}[(i)]
%   \item $\sum_{i,j = 1}^3 b_{ij} y_i y_j \ge k_1 \sum_{i=1}^3 y_i^2$
%     for all $y_i \in \reals$. 
% \item $| b_{ij} | \le k_2 k_1$. 
% \item $ \sum_{i=1}^3 |a_i|^2 + k_1 | g|  \le  k_1^2 k_3.$
%   \end{enumerate}
Then there exists a unique solution $c(x,t)$ of
\eqref{parabolic-PDE-expanded} and a positive constant $C$
independent of $u$ so that   
$$
\| c \|_{L^2(\Omega)} \le C \| u \|_{L^2(\Omega)}.
$$
\end{theorem}
This result along with the fact that $c$ depends linearly on $u$
allows us to define a bounded linear operator $S: L^2(\Omega) \to L^2(\Omega)$
so that $S(u) = c$ whenever the conditions of Theorem
\ref{parabolic-existence-uniqueness} are satisfied. 

Now consider positive constants $r$ and $\tau$ and a sequence of
points $x_i \in D$ and $t_i \in (0,T)$ for $i = 1,2, \cdots, m$. 
Define the sets 
$$ Q_i = B(x_i, r) \times [t_i, t_i +
\tau],$$
where $B(x_i, r) \subset D$ is the ball of radius $r$ centered
at $x_i$. Suppose that $\Omega$ is large enough so that $Q_i \subset
\Omega$ for all $i$ and also $Q_i \cap Q_j = \emptyset$ if $i \neq
j$. To this end, given a solution $u$ of
\eqref{parabolic-PDE-expanded},
 we define the bounded linear measurement operators 
$$
M_i : L^2(\Omega) \to \reals \qquad M_i(c) := \int_{Q_i} c \: dx dt.
$$
The elaborate construction of the $M_i$ corresponds to a common method
of measurement in the study of deposition of particulate matter where
a number of plastic jars (also known as dust-fall jars) are left in
the field for a given period of time \cite{hosseini-lead-inverse,lushi-inverse}. At the end of this period the
jars are taken to the lab and the concentration of deposited material
in each jar is measured.

Putting everything together we can define the forward map
\begin{equation}
  \label{parabolic-PDE-measurements}
\mcl{G}: L^2(\Omega) \to \reals^m \qquad \mcl{G}(u) := ( M_1(S(u)), \cdots, M_m(S(u)))^T.  
\end{equation}
This operator is bounded and linear since $S$ and the $M_i$ are as
well. Now suppose that $y \in \reals^m$ is the data and consider the
usual additive Gaussian noise model 
\begin{equation} \label{parabolic-PDE-data-model}
y = \mcl{G}(u) + \eta \qquad \text{where} \qquad  \eta \sim \mcl{N}(0,
\sigma^2 \mb{I})\qquad 
\text{and} \qquad \sigma > 0.
\end{equation}
 Now we turn our attention to the construction of a
prior measure for the source term.

Let $X = L^2(\Omega)$ and suppose that we have prior knowledge that
the source is constant in time and it is supported within a smooth domain 
$\tilde{D} \subset  D \subset \reals^3$. For example, the domain
$\tilde{D}$ can denote an industrial site which is known as the main
polluter in an area. We let $\{ \psi_k \}_{k=1}^\infty$ and $\{ \lambda_k\}_{k=1}^\infty$ denote the
eigenfunctions and eigenvalues of the Laplacian on $\tilde{D}$.
% {\NNtodo{Perhaps we need to further clarify the conditions on
%     $\tilde{D}$ to make sure that the eignepairs exist and
%     $\{\lambda_k\} \in \ell^2$}}
 The $\psi_k$ form an
orthonormal basis for $L^2(\tilde{D})$ and so they are a good
candidate for prior construction. Consider the random variable
$$
u = \sum_{k=1}^\infty \gamma_k \xi_k \psi_k
$$
where $\gamma_k = \lambda_k^{-1}$ and $\xi_k \sim \text{Exp(1)}$. Let
$\nu$ denote the probability measure that is induced by this random
variable on $L^2(\tilde{D})$.
 By
Corollary \ref{product-prior-is-convex-iid} this measure is convex.
 Now  we uniquely extend the elements of
$L^2(\tilde{D})$ 
by zero to $L^2(D)$ and then extend them as 
constant functions in
direction of $t$ to elements of $L^2(\Omega)$. Let $E:L^2(\tilde{D})
\to L^2(\Omega)$ denote this extension operator which is both bounded
and linear. We now define our prior measure $\mu_0$ to be the push-forward of the measure $\nu$ 
under 
$E$ which is a probability measure on  $L^2(\Omega)$. This measure is convex following Corollary
\ref{convexity-preservation}(i). Putting this fact together with the
forward model given by \eqref{parabolic-PDE-measurements} 
and \eqref{parabolic-PDE-data-model} as well as Theorem
\ref{forward-model-assumptions} and Corollary
\ref{convex-bayesian-inverseproblems-poly-bounded} implies that the Bayesian
inverse problem of finding $u$ from linear measurements of $c$ is well-posed.
\subsection{Example 6: Estimating the permeability in porous medium flow}
In this section we consider an example problem which involves
estimating the coefficients of an elliptic PDE.
 We present this example in the 
context of flow in a porous medium, such as 
groundwater flow.
 The inverse problem
involves estimating the permeability of the porous medium from
measurements of pressure at multiple points. This example
was considered in \cite{dashti-besov} with a Besov prior and in 
\cite{dashti-elliptic-UQ} with Gaussian priors.
Here we will present the same example using a convex prior that is
based on the Gamma distribution. This is an example of a nonlinear inverse problem with a convex prior measure where the forward map does not 
satisfy condition (i) of Theorem \ref{forward-model-assumptions}
for every $\epsilon > 0$. Therefore, we have to take extra care to make
sure that the prior results in a well-posed inverse problem.

Consider the elliptic PDE 
\begin{equation}
  \label{elliptic-PDE}
  - \nabla \cdot ( \exp( u(x) ) \nabla p(x) )  = f ,\qquad x \in
  \mbb{T}^2, \qquad f \in L^2(\mbb{T}^2),
\end{equation}
where the boundary conditions are periodic and $\mbb{T}^2 = (0,1]^2$. Here $\exp(u(x))$ is the permeability of the medium. We choose
to work with the exponential form to ensure that the permeability is positive.
Assume that the data $y \in \reals^m$ consists of noisy
pointwise measurements of the pressure $p(x)$ (\myhl{alternatively,
  one can consider local averages of the pressure field if $p$ is
  not defined pointwise}) i.e. 
$$
 y = \mcl{G}(u) + \eta \qquad \text{where} \qquad \eta \sim
 \mcl{N}(0,\sigma^2 \mb{I}) \qquad \text{and} \qquad \sigma >0
$$
and 
\begin{equation}\label{elliptic-PDE-forward-map-definition}
\mcl{G} (u) := ( p(x_1), \cdots, p(x_m) )^T
\qquad \text{where}
 \qquad x_1, x_2, \cdots, x_m \in \mbb{T}^2.
\end{equation}
The noise variance $\sigma$ and the points $\{ x_k \}_{k=1}^m$ are fixed. 

We are interested in the setting where $u \in
C^1(\mbb{T}^2)$.
Under this assumption we have the following theorem concerning
the boundedness and continuity of the forward map.
\begin{theorem}[{\cite[Corollary~3.5]{dashti-elliptic-UQ}}]\label{elliptic-PDE-forward-map-properties}
Suppose that $\mcl{G}$ is given by
\eqref{elliptic-PDE-forward-map-definition} where $p(x)$ is
given by \eqref{elliptic-PDE} and $f \in L^2(\mbb{T}^2)$. Then for any
$u \in C^1(\mbb{T}^2)$ there exists a constant $D_1 = D_1( m, \|f
\|_{L^2(\mbb{T})})$ such that 
$$
\|\mcl{G}(u) \|_2  \le D_1 \exp( \| u \|_{C^1(\mbb{T}^2)}).
$$ 
If $u_1, u_2 \in C^1(\mbb{T}^2)$, then for any
$\epsilon > 0$
$$
\| \mcl{G}(u_1) - \mcl{G}(u_2) \|_2 \le D_2 \exp( c \max \{ \| u_1
\|_{C^1(\mbb{T}^2)} ,\| u_2 
\|_{C^1(\mbb{T}^2)} ) \| u_1 - u_2 \|_{C^1(\mbb{T}^2)},
$$
where $D_2= D_2( M, \epsilon, \|f \|_{L^2(\mbb{T}^2)})$ and $c = 4 + 8 + \epsilon$.
\end{theorem}

This theorem suggests that
we need to construct a prior that is supported on
$C^1(\mbb{T}^2)$. Unfortunately, since $C^1(\mbb{T}^2)$ does not have an
unconditional basis our recipe for construction of the prior measures
from Section \ref{construction-of-convex-priors} does not apply directly. Instead, we shall construct
the prior to be supported within a Sobolev space that is sufficiently
regular and we use the Sobolev embedding theorem to show
that 
the prior will be supported on $C^1(\mbb{T}^2)$ as desired. 

Let $\{ \psi_j \}_{j=1}^\infty$ be an $r$-regular wavelet basis for
$L^2(\mbb{T}^2)$ \cite[Section~2.1]{meyer} where $r > 2$. Then for $f
\in L^2(\mbb{T}^2)$ we can write 
$$
f(x) = \sum_{k=1}^\infty  \hat{f}_k \psi_k(x)
$$
where $\{ \hat{f}_k\}$ is the sequence of basis coefficients of $f$
(see \cite[Appendix~A]{lassas-invariant} for how the wavelet basis indices
are reordered to form a single sum). For $s<
r$ we can
identify the Sobolev space $H^s(\mbb{T}^2)$ as 
$$
H^s(\mbb{T}^2) := \left\{ f \in L^2(\mbb{T}^2) : \sum_{k=1}^\infty k^s
  | \hat{f}_k |^2 <\infty \right\}.
$$

As usual, we construct the prior measure by randomizing the basis
coefficients. Let $\{ \xi_k\}$ and $\{\zeta_k\}$ be two sequences of
i.i.d.\  random variables on $\reals$ that are distributed according to
$\text{Gamma}(2,1)$ and define $\theta_k = \xi_k - \zeta_k$. This
construction ensures that $\theta_k$ has a symmetric distribution. By
Corollary \ref{convexity-preservation}(i) and (ii) the $\theta_k$ are
convex. Now 
 consider the random variable 
 \begin{equation}
   \label{wavelet-prior-sample}
v = \sum_{k=1}^\infty \gamma_k \theta_k \psi_k   
 \end{equation}
where $\gamma_k = k^{-2}$. By Corollary \ref{product-prior-is-convex-iid}
we know that $\| v\|_{L^2(\mbb{T}^2)} < \infty$ a.s. Furthermore, 
$$
\sum_{k=1}^\infty k^3 |\gamma_k|^2 |\eta_k|^2 = \sum_{k=1}^\infty k^{-1}
|\eta_k|^2.
$$
But this sum converges almost surely by Kolmogorov's two series
theorem and dominated convergence and so $\| v \|_{H^3(\mbb{T}^2)} < \infty$ a.s. By the Sobolev embedding theorem \cite[Proposition~3.3]{taylor-PDE} $H^3(\mbb{T}^2) \subset
C^1(\mbb{T}^2)$ and so the prior measure induced by the random variable
$v$ in \eqref{wavelet-prior-sample} is supported in $C^1(\mbb{T}^2)$ as
desired and it is a convex measure.

Before we proceed to proving the well-posedness of this problem 
 we will need to reweight the 
samples $v$. The reason for this issue is that 
$ \EE \exp( \| v \|_X) $ may not be \myhl{finite}. However, by Theorem 
\ref{convex-measurable-exp} we know that there exists a constant 
$\kappa > 0$ so that $\EE \exp ( \kappa \| v \|_X ) < \infty$. 
Thus, we take the prior samples $u = \beta v$ for some $\beta \in
(0, \kappa]$ which ensures that $\EE \exp( \| u\|_X) < \infty$.
 We
are now able to apply Theorems \ref{elliptic-PDE-forward-map-properties},
\ref{forward-model-assumptions} and
\ref{convex-bayesian-inverseproblems} to the prior measure that is
induced by $u$
 in order to show that 
this inverse problem is well-posed.

Note that estimating the constant $\kappa$ for a given prior is 
in general a difficult task and to the best of our knowledge a general
recipe for estimating this constant does not exist in the
literature. Of course the actual value of this constant is only
important when the forward map is exponentially bounded by the
parameter 
norm. For example, Theorem~\ref{elliptic-PDE-forward-map-properties}
above  dictates that $\exp( \kappa \| u\|_{C^1(\mbb{T}^2)} )$ must be
integrable under the prior for $\kappa >1$ in order for us to achieve well-posedness. Thus, the properties of the forward map 
can be used to identify the minimum ``allowed'' value of the constant
$\kappa$.
On the other hand,
if the forward map is polynomially bounded then well-posedness
can be \myhl{achieved} regardless of the actual value of $\kappa$ according to
Corollary \ref{convex-bayesian-inverseproblems-poly-bounded}
% \NNtodo{Can this $\kappa]$ actually be computed for samples like in 4.12? Or estimated based on the random variables you're using for the coefficients?}

% that the Bayesian inverse
% problem of finding $u(x)$ from pointwise measurements of $p(x)$
% the forward map given by \eqref{elliptic-PDE} and
% \eqref{elliptic-PDE-forward-map-definition}
%is well-posed.
%\subsection{Connection to convex regularization}
\section{Closing remarks}
We started this article by defining the notions of well-posedness
and consistency of Bayesian inverse problems in the general setting
where the parameter of interest belongs to an infinite-dimensional
Banach space. We presented a set of assumptions on the prior measure $\mu_0$
and the likelihood potential $\Phi$ under which the resulting
inverse problem would be well-posed. \myhl{Furthermore, we discussed consistent approximation of the 
posterior measure via an approximation $\Phi_N$ of the likelihood potential $\Phi$. 
We discussed mild conditions
on the forward map and the likelihood potential that allowed us to
obtain bounds on the rate of convergence of the
approximate posterior in the Hellinger metric. We particularly focused on the setting where 
$\Phi_N$ is obtained by discretizing the forward problem using finite dimensional projections. 
}

\myhl{
We mainly
focused on our assumptions concerning the prior measure $\mu_0$} and
showed that the class of convex measures is a
good choice for modelling of prior knowledge as its elements
result in well-posed inverse problems. This class already
includes well known measures such as Gaussian and Besov measures and so 
our results can be viewed as a generalization of existing results
regarding well-posedness of Bayesian inverse problems. 

Afterwards, we presented a general framework for the construction of priors
 on separable Banach spaces that have an unconditional
basis. Inspired
by the Karhunen--Lo{\'e}ve expansion of Gaussian random variables, our
framework uses a countable product of one
dimensional convex measures on $\reals$. Next, we considered some of
the practical aspects of solving the Bayesian inverse problems that
arise from choosing convex priors. Finally, we presented
four concrete examples of well-posed Bayesian inverse problems that 
used convex prior measures.

An important consequence of this work is that
now we have access to a much larger class of measures for modelling of
prior knowledge in Bayesian inverse problems. For example, if one is
interested in imposing a constraint such as positivity then one can
use the Gamma distribution or the uniform distribution to do so. The
resulting measure will still be convex. More interestingly, recent
results in \cite{scott-shrink} and \cite{calvetti-hierarchical} hint
at the use of specific convex measures in order to promote sparsity of
the parameters. Even in the case of the Besov priors of \cite{lassas-invariant}, the resulting
maximum a posteriory estimator is equivalent to solving a least
squares problem with an $\ell^1$ regularization term which is a common
technique for promoting sparsity of the solution. Then the
well-posedness result for the class of convex measures is a first step
towards the study of sparse solutions within the Bayesian approach
to inverse problems.
% \begin{itemize}
% \item definition 3.1 and Theorem 3.1 extend to locally convex Hausdorff vector spaces.
% \item Theorem 2.2 holds when X is a Frechet space.
% \item Theorem 3.3 is true when X is locally convex.
% \end{itemize}
\section*{Acknowledgements}
We are thankful to Dr. Sergios Agapiou for pointing us towards the
works of Christer Borell on
convex measures. 
We are also indepted to Profs. Paul Tupper and Tim Sullivan for
numerous comments and insightful discussions.

\bibliographystyle{abbrv}
%\begin{thebibliography}{99}
\bibliography{ref}

\begin{thebibliography}{10}

\bibitem{agapiou}
S.~Agapiou, J.~M. Bardsley, O.~Papaspiliopoulos, and A.~M. Stuart.
\newblock Analysis of the {G}ibbs sampler for hierarchical inverse problems.
\newblock {\em SIAM/ASA Journal on Uncertainty Quantification}, 2(1):511--544,
  2014.

\bibitem{bagnoli-log-concave}
M.~Bagnoli and T.~Bergstrom.
\newblock Log-concave probability and its applications.
\newblock {\em Economic Theory}, 26(2):445--469, 2005.

\bibitem{calvetti-hierarchical}
J.~M. Bardsley, D.~Calvetti, and E.~Somersalo.
\newblock Hierarchical regularization for edge-preserving reconstruction of
  {PET} images.
\newblock {\em Inverse Problems}, 26(3):035010, 2010.

\bibitem{bernardo}
J.~M. Bernardo and A.~F. Smith.
\newblock {\em Bayesian Theory}.
\newblock Wiley Series in Probability and Statistics. John Wiley \& Sons, New
  York, 2009.

\bibitem{bogachev-gaussian}
V.~I. Bogachev.
\newblock {\em {Gaussian Measures}}, volume~62 of {\em {Mathematical Surveys
  and Monographs}}.
\newblock {American Mathematical Society}, Providence, 1998.

\bibitem{bogachev1}
V.~I. Bogachev.
\newblock {\em Measure Theory}, volume~1.
\newblock Springer, New York, 2007.

\bibitem{bogachev2}
V.~I. Bogachev.
\newblock {\em Measure Theory}, volume~2.
\newblock Springer, New York, 2007.

\bibitem{bogachev-malliavin}
V.~I. Bogachev.
\newblock {\em Differentiable measures and the Malliavin calculus}, volume 164
  of {\em Mathematical Sureverys and Monographs}.
\newblock American Mathematical Society, Providence, 2010.

\bibitem{borell-convex}
C.~Borell.
\newblock Convex measures on locally convex spaces.
\newblock {\em Arkiv f{\"o}r Matematik}, 12(1):239--252, 1974.

\bibitem{burger-map}
M.~Burger and F.~Lucka.
\newblock Maximum a posteriori estimates in linear inverse problems with
  log-concave priors are proper {B}ayes estimators.
\newblock {\em Inverse Problems}, 30(11):114004, 2014.

\bibitem{calvetti}
D.~Calvetti and E.~Somersalo.
\newblock {\em {An Introduction to {B}ayesian Scientific Computing: Ten
  Lectures on Subjective Computing}}, volume~2 of {\em Surveys and Tutorials in
  the Applied Mathematical Sciences}.
\newblock Springer Science \& Business Media, New York, 2007.

\bibitem{scott-horseshoe}
C.~M. Carvalho, N.~G. Polson, and J.~G. Scott.
\newblock The horseshoe estimator for sparse signals.
\newblock {\em Biometrika}, page asq017, 2010.

\bibitem{cotter-approximation}
S.~L. Cotter, M.~Dashti, and A.~M. Stuart.
\newblock Approximation of {B}ayesian inverse problems for {PDE}s.
\newblock {\em SIAM Journal on Numerical Analysis}, 48(1):322--345, 2010.

\bibitem{stuart-mcmc}
S.~L. Cotter, G.~O. Roberts, A.~M. Stuart, and D.~White.
\newblock {MCMC} methods for functions: modifying old algorithms to make them
  faster.
\newblock {\em Statistical Science}, 28(3):424--446, 2013.

\bibitem{dashti-besov}
M.~Dashti, S.~Harris, and A.~M. Stuart.
\newblock Besov priors for {B}ayesian inverse problems.
\newblock {\em Inverse Problems and Imaging}, 6(2):183--200, 2012.

\bibitem{dashti-elliptic-UQ}
M.~Dashti and A.~M. Stuart.
\newblock Uncertainty quantification and weak approximation of an elliptic
  inverse problem.
\newblock {\em SIAM Journal on Numerical Analysis}, 49(6):2524--2542, 2011.

\bibitem{stuart-bayesian-lecture-notes}
M.~Dashti and A.~M. Stuart.
\newblock The {B}ayesian approach to inverse problems.
\newblock arXiv preprint:1302.6989, 2015.

\bibitem{diaconis-consistency}
P.~Diaconis and D.~Freedman.
\newblock On the consistency of {Bayes} estimates.
\newblock {\em The Annals of Statistics}, 14(1):1--26, 1986.

\bibitem{Evans}
L.~C. Evans.
\newblock {\em Partial differential equations}.
\newblock Number~19 in Graduate Studies in Mathematics. American Mathematical
  Society, Providence, 2010.

\bibitem{foucart}
S.~Foucart and H.~Rauhut.
\newblock {\em A mathematical introduction to compressive sensing}.
\newblock Applied and Numerical Harmonic Analysis. Springer Sience \& Business
  Media, New York, 2013.

\bibitem{freedman-consistency}
D.~A. Freedman.
\newblock On the asymptotic behavior of {B}ayes' estimates in the discrete
  case.
\newblock {\em The Annals of Mathematical Statistics}, 34(4):1386--1403, 1963.

\bibitem{hansen-deblurring}
P.~C. Hansen, J.~G. Nagy, and D.~P. O'leary.
\newblock {\em Deblurring images: matrices, spectra, and filtering}.
\newblock {SIAM}, Philadelphia, 2006.

\bibitem{heil-basis}
C.~Heil.
\newblock {\em A basis theory primer: expanded edition}.
\newblock Applied and Numerical Harmonic Analysis. Springer Sicence \& Business
  Media, New York, 2010.

\bibitem{helin-MAP}
T.~Helin and M.~Burger.
\newblock Maximum a posteriori probability estimates in infinite-dimensional
  {B}ayesian inverse problems.
\newblock {\em Inverse Problems}, 31(8):085009, 2015.

\bibitem{hosseini-dispersion}
B.~Hosseini.
\newblock Dispersion of pollutants in the atmosphere: {A} numerical study.
\newblock Master's thesis, Department of Mathematics, Simon Fraser University,
  2013.
\newblock Available online at: http://summit.sfu.ca/item/13646.

\bibitem{hosseini-lead-inverse}
B.~Hosseini and J.~M. Stockie.
\newblock Bayesian estimation of airborne fugitive emissions using a gaussian
  plume model.
\newblock {\em Atmospheric Environment}, 141:122--138, 2016.

\bibitem{iglesias-geometric}
M.~A. Iglesias, K.~Lin, and A.~M. Stuart.
\newblock Well-posed {B}ayesian geometric inverse problems arising in
  subsurface flow.
\newblock {\em Inverse Problems}, 30(11):114001, 2014.

\bibitem{somersalo}
J.~Kaipio and E.~Somersalo.
\newblock {\em {S}tatistical and {C}omputational {I}nverse {P}roblems}, volume
  160 of {\em {A}pplied {M}athematical {S}ciences}.
\newblock Springer Sience \& Business Media, New York, 2005.

\bibitem{kallenberg}
O.~Kallenberg.
\newblock {\em Foundations of modern probability}.
\newblock Probability and Its Application. Springer, New York, 2006.

\bibitem{lassas-sparsity-promoting}
V.~Kolehmainen, M.~Lassas, K.~Niinim{\"a}ki, and S.~Siltanen.
\newblock Sparsity-promoting {B}ayesian inversion.
\newblock {\em Inverse Problems}, 28(2):025005, 2012.

\bibitem{lassas-invariant}
M.~Lassas, E.~Saksman, and S.~Siltanen.
\newblock Discretization-invariant {B}ayesian inversion and {B}esov space
  priors.
\newblock {\em Inverse Problems and Imaging}, 3(1):87--122, 2009.

\bibitem{lushi-inverse}
E.~Lushi and J.~M. Stockie.
\newblock An inverse {G}aussian plume approach for estimating atmospheric
  pollutant emissions from multiple point sources.
\newblock {\em Atmospheric Environment}, 44(8):1097--1107, 2010.

\bibitem{meyer}
Y.~Meyer.
\newblock {\em Wavelets and operators}, volume~37 of {\em Cambridge Studies in
  Advanced Mathematics}.
\newblock Cambridge University Press, Cambridge, 1992.

\bibitem{scott-shrink}
N.~G. Polson and J.~G. Scott.
\newblock Shrink globally, act locally: Sparse {B}ayesian regularization and
  prediction.
\newblock {\em {B}ayesian Statistics}, 9:501--538, 2010.

\bibitem{Renardy}
M.~Renardy and R.~C. Rogers.
\newblock {\em An introduction to partial differential equations}.
\newblock Number~13 in Texts in Applied Mathematics. Springer, New York, 2010.

\bibitem{seinfeld}
J.~H. Seinfeld and S.~N. Pandis.
\newblock {\em {Atmospheric Chemistry and Physics: From Air Pollution to
  Climate Change}}.
\newblock {John Wiley \& Sons}, 1997.

\bibitem{stuart-acta-numerica}
A.~M. Stuart.
\newblock Inverse problems: a {B}ayesian perspective.
\newblock {\em Acta Numerica}, 19:451--559, 2010.

\bibitem{stuart-GP}
A.~M. Stuart and A.~L. Teckentrup.
\newblock Posterior consistency for {G}aussian process approximations of
  {B}ayesian posterior distributions.
\newblock 2016.
\newblock arXiv preprint:1603.02004.

\bibitem{sullivan}
T.~J. Sullivan.
\newblock Well-posed {B}ayesian inverse problems and heavy-tailed stable
  {B}anach space priors.
\newblock 2016.
\newblock arXiv preprint:1605.05898.

\bibitem{taylor-PDE}
M.~E. Taylor.
\newblock {\em Partial Differential Equations I: Basic Theory}, volume 115 of
  {\em Applied Mathematical Sciences}.
\newblock Springer Science \& Business Media, New York, second edition, 2011.

\bibitem{vogel}
C.~R. Vogel.
\newblock {\em {C}omputational {M}ethods for {I}nverse {P}roblems}.
\newblock {SIAM}, Philadelphia, 2002.

\end{thebibliography}
%\include{ref.bib}

%\end{thebibliography}

\end{document}